\documentclass[12pt]{article}
 
\usepackage[margin=1in]{geometry}  % set the margins to 1in on all sides
\usepackage{graphicx}              % to include figures
\usepackage{amsmath, enumitem}               % great math stuff
\usepackage{amsfonts}              % for blackboard bold, etc
\usepackage{amsthm}                % better theorem environments
\usepackage{amssymb}
\usepackage{bbm}
\usepackage[square, sort, comma, numbers]{natbib}
\usepackage{url}
\usepackage[autostyle]{csquotes}
 \usepackage[linktocpage=true]{hyperref}
\usepackage{tocloft}
\usepackage{subfig}

\newtheorem{thm}{Theorem}[]
\newtheorem{lem}[thm]{Lemma}
\newtheorem{prop}[thm]{Proposition}
\newtheorem{cor}[thm]{Corollary}

\newtheorem*{remark}{Remark}
\newtheorem*{exa}{Example}

\newenvironment{definition}[1][Definition.]{\begin{trivlist}
\item[\hskip \labelsep {\bfseries #1}]}{\end{trivlist}}

\DeclareMathOperator{\id}{id}

\newcommand{\clr}{\operatorname{clr}} 
\newcommand{\inv}{\operatorname{inv}} 
\newcommand{\sfm}{\operatorname{sfm}} 
 
\newcommand{\ilr}{\operatorname{ilr}} 
 
\newcommand{\diag}{\operatorname{diag}}

\newcommand\invisiblesection[1]{%
  \refstepcounter{section}%
  \addcontentsline{toc}{section}{\protect\numberline{\thesection}#1}%
  \sectionmark{#1}}

     \DeclareMathOperator*{\argmin}{argmin}
      \DeclareMathOperator*{\argmax}{argmax}

\begin{document}

\title{Darwinian evolution as Brownian motion on the simplex:\\  {\Large A geometric perspective on stochastic replicator dynamics}}

\author{Tobias Lehmann
\\
University of Leipzig} 
       
\maketitle  
    
\begin{abstract}
We prove that stochastic replicator dynamics can be interpreted as intrinsic Brownian motion on the simplex equipped the Aitchison geometry. As an immediate consequence we derive three approximation results in the spirit of Wong-Zakai approximation, Donsker's invariance principle and a JKO-scheme. Finally, using the Fokker-Planck equation and Wasserstein-contraction estimates, we study the long time behavior of the stochastic replicator equation, as an example of a non-gradient drift diffusion on the Aitchison simplex. 

\end{abstract}
\bigskip
       \begin{center}
C{\footnotesize ONTENTS}
\end{center}  
\vspace{-1.2cm}
\tableofcontents         
          
\newpage         
\invisiblesection{Introduction and outline}           
\begin{center}
1. I{\footnotesize NTRODUCTION AND OUTLINE}
\end{center}
The Aitchsion geometry is a Hilbert space structure on the open standard unit simplex and used prominently in compositional data analysis. The aim of this note is to present a seemingly new and interesting connection between Darwinian evolution modeled through stochastic replicator equations on one hand and Brownian motion on the Aitchison simplex on the other hand.\par

Let us elaborate this objective. The classical way, first introduced in \citep{taylor1978evolutionary}, to reformulate Darwin's paradigm of selection in mathematical language is by means of replicator dynamics. Consider a population with $n$ distinct types (e.g. genotypes) and denote by $p_{i}(t)$ the share of individuals with type $i$ at time $t$. Also, given a fitness landscape $f=(f^{1},\dots,f^{n})\colon\mathbb{R}^{n}\to\mathbb{R}^{n}$ we write
\begin{equation*}
\bar{f}(p)=p\cdot f(p)=\sum_{i}^{n}p_{i}f^{i}(p),\quad i=1,\dots,n
\end{equation*}
for the mean fitness. Now fix some initial datum $p\in\Delta$, where  
\begin{equation*}
\Delta=\Delta_{n}:=\lbrace p\in (0,1)^{n}\ |\ p_{1}+\dots+p_{n}=1 \rbrace
\end{equation*}
is the open standard unit simplex. Then the replicator equation reads 
\begin{equation}\label{eq: rep}
\dot{p}_{i}(t)=p_{i}(t)\left(f^{i}(p(t))-\bar{f}(p(t))\right),\quad p(0)=p
\end{equation} 
and models an evolution of type compositions $p(t)=(p_{1}(t),\dots,p_{n}(t))\in\Delta$ undergoing selection through the fitness landscape $f$.\par  
Especially well studied is the situation of linear fitness landscapes, i.e. when
${f(p)=Ap}$ for some payoff matrix $A\in\mathbb{R}^{n\times n}$, where profound connections to evolutionary game theory and Smith's concept of evolutionary stable strategies arise \citep[c.f.][]{hofbauer1998evolutionary}. But also the simple scenario in which the fitness landscape is frequency independent and thus given by a vector ${f=(f^{1},\dots,f^{n})\in\mathbb{R}^{n}}$ is of interest. In this situation (\ref{eq: rep}) describes the prebiotic evolution of self-replicating polynucleotides (e.g RNA, DNA) without mutations \citep[c.f.][]{sigmund1986survey, nowak2006evolutionary} and is a particular example of a class of dynamical systems introduced by Eigen and Schuster in their theories of quasispecies and hypercycles \citep{eigen1971selforganization, eigen2012hypercycle}.\par 
Often, allowing in mathematical models for uncertainty or randomness leads to a description better fitting empirical evidence. Applying this principle to a replicator equation with linear fitness landscape given by the matrix $A$ leads to the stochastic replicator equation
\begin{equation}\label{eq:fudsde}
dX_{t}=b(X_{t})dt+\sigma(X_{t})dB_{t},
\end{equation}
where $B$ is an $n$-dimensional Brownian motion, the drift component is given by 
\begin{equation*}
b(x)=\left(\diag(x_{1},\dots,x_{n})-x\otimes x\right)\left(A-\diag\left(\sigma_{1}^{2},\dots,\sigma_{n}^{2}\right)\right)x
\end{equation*}
and the diffusion matrix obeys
\begin{equation*}
\sigma(x)=\left(\diag(x_{1},\dots,x_{n})-x\otimes x\right)\diag\left(\sigma_{1},\dots,\sigma_{n}\right).
\end{equation*}
Initially proposed in \citep{fudenberg1992evolutionary}, this model and in particular its long-time behavior attracted a lot of interest over the last decades. We refer exemplarly to \citep{hofbauer2009time} and references therein.\par 
 Although we will not take this perspective here, we want emphasize that apart from the biological application, there is also an interpretation for (\ref{eq:fudsde}) in terms of mathematical finance. Indeed, using the language of Fernholz' stochastic portfolio theory \citep[c.f.][]{fernholz2002stochastic}, $X$ can be seen as the evolution of a market portfolio, for which the rates of return $r_{i}$ of the underlying stock prices experience feedback through $X$ via $r_{i}(t)=(AX_{t})_{i}$. 
 
\par\medskip 
In the second section of this note we recall basic principles of the Aitchison geometry. In short, we will see that, when equipped with appropriate vector space operations, the simplex can be given a Hilbert space structure by dint of the inner product
\begin{equation*}
\langle p,q\rangle_{A}:=\frac{1}{2n}\sum_{i,j=1}^{n}\ln\frac{p_{i}}{p_{j}}\ln\frac{q_{i}}{q_{j}}.
\end{equation*}  

Now consider (\ref{eq:fudsde}) with $A=0$ and $\sigma_{i}=1$ for $i=1,\dots,n$. The corresponding diffusion process, that we refer to as the \textit{Aitchison diffusion}, evolves according to the Stratonovich SDE
 \begin{equation}\label{eq:aidif}
\begin{aligned}
dX^{i}_{t}&=X^{i}_{t}\left(\circ dB^{i}_{t}-\sum_{j=1}^{n}X^{j}_{t}\circ dB^{j}_{t}\right),\quad i=1,\dots n\\
X_{0}&=p\in\Delta,
\end{aligned}
\end{equation}
which can be interpreted as a replicator equation in a white noise fitness landscape. Then the main result of this note, which we present in Section 3, asserts that  the Aitchison diffusion is nothing but Brownian motion on $(\Delta,\langle\cdot,\cdot\rangle_{A})$.\par 
This observation directely entails three approximation results to be presented in \linebreak Section 4. Apart from a Wong-Zakai approximation and a JKO-scheme for the associated heat equation on the simplex, we show that discrete stochastic replicator dynamics are random walks on the Aitchison simplex which, in the spirit of Donsker's invariance principle, can be used to approximate (\ref{eq:aidif}).\par 

In the last section we study the stochastic replicator dynamic as an example of a drift diffusion processes on $\Delta$. Using the associated Fokker-Planck equation, we investigate the long-time behavior and recover results proven earlier in \citep{hofbauer2009time}. The final subsection is devoted to questions of the relaxation to equilibrium for replicator dynamics which we analyse by means of Wasserstein contraction estimates. As a major result we characterize those payoff matrices that steer trajectories of the stochastic replicator equation to synchronize with a Langevin dynamic on $\Delta$ driven by cross-entropy.

\invisiblesection{A primer on the Aitchison geometry of the simplex}   
  
 \medskip  
 \begin{center}
2. A {\footnotesize  PRIMER ON THE AITCHISON GEOMETRY OF THE SIMPLEX}
\end{center} 
  Often, most notably in geology and chemistry but also in ecology or social sciences, one is confronted with data which represents portions of a total. We may think of the chemical composition of 100 soil samples taken at different places in Germany. Many classical statistical methods are relying on Euclidean geometry and are therefore inappropriate for analysing data constrained to a constant total sum. Traditionally, such data is called \textit{compositional data} (CoDa) and the corresponding branch of statistics \textit{compositional data analysis}, for which a plenty of literature is available \citep[c.f.][]{10.2307/2345821, pawlowsky2007lecture,  van2013analyzing, codaweb}. In the following we will mainly rely on \citep{pawlowsky2007lecture}. \par 
One of the most influential developments in the history of compositional data analysis was Aitchison's idea \citep{Aitchison1984} of equipping the simplex with a Hilbert space structure, which in his honour is nowadays called \textit{Aitchison geometry} and defined as follows:\par\medskip

\noindent For any $p,q\in\mathbb{R}_{> 0}$ and $\alpha\in\mathbb{R}$ we define the operations:\par\medskip
the \textit{closure} of $p$ 
   \begin{equation*}
   \mathcal{C}(p):=\left(\sum_{j=1}^{n}p_{j}\right)^{-1}p,
   \end{equation*}
\indent the \textit{perturbation} of $p$ by $q$
\begin{equation*}
p\oplus q:=\mathcal{C}(p_{1}q_{1},\dots,p_{n}q_{n}),
\end{equation*}
\indent the \textit{powering} of $p$ by $\alpha$
\begin{equation*}
\alpha\odot p:=\mathcal{C}(p_{1}^{\alpha},\dots,p_{n}^{\alpha}).
\end{equation*}
Then, one can easily check that $(\Delta,\oplus,\odot)$ is a vector space, in which the neutral element $e$ is the barycenter of $\Delta$, i.e.
\begin{equation*}
e=\mathcal{C}(1,\dots,1) 
\end{equation*}
and where the inverse of $p\in\Delta$
is given by
\begin{equation*} 
\inv(p):=p^{-1}:=-1\odot p=\mathcal{C}(p_{1}^{-1},\dots,p_{n}^{-1}).
\end{equation*}
Moreover, if we introduce the the Aitchison inner product 
\begin{equation}\label{eqn: Aitprod}
\langle p,q\rangle_{A}:=\frac{1}{2n}\sum_{i,j=1}^{n}\ln\frac{p_{i}}{p_{j}}\ln\frac{q_{i}}{q_{j}}
\end{equation}
then one can show
 
\begin{thm}[Hilbert space structure of $\Delta$, {\citep[c.f.][]{10.2307/3085883}}]\label{thm: Aihi}
The vector space $(\Delta,\oplus,\odot)$ equipped with the inner product (\ref{eqn: Aitprod}) is a Hilbert space.
\end{thm}
 
Theorem \ref{thm: Aihi} is easily justify by providing a Hilbert space isomorphism. The most prominent example, which is of major importance in both, compositional data analysis as well as throughout this note, is the so called \textit{centered log-ratio transform} which maps $\Delta$ to the Hilbert space $(H,(\cdot,\cdot)_{H})$, where
\begin{equation}\label{eq:H}
H=H_{n}:=\left\lbrace x\in\mathbb{R}^{n}\ \ | \ x_{1}+\dots+x_{n}=0\ \right\rbrace
\end{equation}
and $(\cdot,\cdot)_{H}$ is the scalar product inherited from the standard inner product on $\mathbb{R}^{n}$. This transform is defined as

\begin{equation*}
\clr\colon\Delta\to H,\quad p\mapsto\clr(p):=\left(\ln\frac{p_{i}}{g(p)}\right)_{i=1}^{n},
\end{equation*}
where $g(p)=\left(\prod_{i=1}^{n}p_{i}\right)^{\frac{1}{n}}$ is the geometric mean of $p$. As a consequence  we obtain the usual transformation rules
\begin{equation}\label{eqn: traforule}
(i)\quad \clr(\alpha\odot p\oplus q)=\alpha\clr(p)+\clr(q)\quad\quad\text{ and } \quad\quad(ii)\quad  \langle p,q\rangle_{A}=(\clr(p),\clr(q))_{H}.
\end{equation}
For $p,q\in\Delta$ we will abbreviate $p\oplus(-1)\odot q=:p\ominus q$. Then, the Aitchison distance on $\Delta$ induced through $\langle\cdot,\cdot\rangle_{A}$ obeys
\begin{equation}\label{eqn: Aidist}
d^{2}_{A}(p,q):=\langle p\ominus q,p\ominus q\rangle=\frac{1}{2n}\sum_{i,j=1}^{n}\left(\ln\frac{p_{i}}{p_{j}}-\ln\frac{q_{i}}{q_{j}}\right)^{2}.
\end{equation} 

Because it will be of frequent use later on, we will denote the inverse log-ration transform
\begin{equation*}
\clr^{-1}\colon H\to\Delta,\qquad x\mapsto\clr^{-1}(x)=\mathcal{C}(e^{x_{1}},\dots,e^{x_{n}}),
\end{equation*} 
by $\sfm$. The reason for doing so, is that
\begin{equation*}
\sfm_{i}(x):=\clr^{-1}_{i}(x)=\frac{e^{x_{i}}}{\sum_{j}e^{x_{j}}},\quad i=1,\dots,n
\end{equation*}
is an ubiquitous object in applied mathematics. Whereas it (or close variants of it) occurs as Boltzmann or Gibbs distribution in statistical mechanics, it is a well known map also in evolutionary game theory and decision theory. Over the last decades it has been used most prominently in machine learning, where it is referred to as softmax function \citep[c.f.][and references therein]{gao2017properties}, which is the motivation for our naming. 

\begin{remark}
We defined $\clr\colon\Delta\to H$ and $\sfm\colon H\to\Delta$ for the sake of bijectivity, but of course a priori $\clr$ and $\sfm$ are well defined also for arguments in $\mathbb{R}^{n}_{>0}$ and $\mathbb{R}^{n}$, respectively. 
\end{remark}

Although the $\clr$ transform is easy to compute, it has the drawback of mapping to $H$, whereas often one would rather prefer an isomorphism realising $\Delta\cong\mathbb{R}^{n-1}$. This can be easily achieved by appropriately post- or pre processing $\clr$ and $\sfm$, respectively. To this end, let $\lbrace e_{1},\dots,e_{n-1}\rbrace$ be an orthonormal base of $\Delta$ and define the \emph{contrast matrix} $\Psi\in\mathbb{R}^{(n-1)\times n}$ by
\begin{equation*}
\Psi:=
 \begin{pmatrix}
  \Psi_{1} \\
  \vdots \\
  \Psi_{n-1}
 \end{pmatrix}
 :=
 \begin{pmatrix}
  \clr{e^{1}} \\
  \vdots \\
  \clr{e^{n-1}}
 \end{pmatrix}.
\end{equation*}
Observe that, independent of the choice of the basis, contrast matrices satisfy
\begin{equation}\label{eq:contrastprop}
\Psi\Psi^{\top}=\operatorname{Id}_{n-1}\qquad \qquad \text{ and } \qquad \qquad \Psi^{\top}\Psi=\operatorname{Id}_{n}-\frac{1}{n}\mathbf{1}_{n}\otimes\mathbf{1}_{n},
\end{equation}
where $\mathbf{1}_{n}=(1,\dots,1)\in\mathbb{R}^{n}$.
Then the \textit{isometric log-ratio transform} $\ilr$ is defined by
\begin{equation*}
\ilr(p):=\Psi\clr(p)
\end{equation*}
with inverse
\begin{equation*}
\ilr^{-1}(x)=\sfm(\Psi^{\top}x). 
\end{equation*} 
Equivalently, we can express $\ilr$ by
\begin{equation*}
\ilr(p)=(\langle p,e^{1}\rangle_{A},\dots,\langle p,e^{n-1}\rangle_{A})^{\top},
\end{equation*} 
from which we can immediately deduce the desirable property
\begin{equation*}
\ilr(e_{i})=\epsilon_{i},
\end{equation*}
where $\lbrace\epsilon_{i}\rbrace_{i=1}^{n-1}$ is the standard basis in $\mathbb{R}^{n-1}$. Henceforth, we will use the symbols $e_{i}$ and $\epsilon_{i}$ generically for basis elements in $\Delta$ and $\mathbb{R}^{n}$, respectively.\par

We continue with a short discussion on the relations between the Euclidean distance $d_{\mathbb{R}^{n}}$ and the Aitchison distance $d_{A}$ on $\Delta$ and the topologies each of the two the distances induces on the simplex.

\begin{lem}\label{lem:topequiv}
The metrics $d_{\mathbb{R}^{n}}$ and $d_{A}$ induce the same topology on $\Delta$. 
\end{lem}

\begin{proof}
We need to prove that $\operatorname{id}\colon\Delta\to\Delta$ is both, $(d_{\mathbb{R}^{n}},d_{A})$-continuous and $(d_{A},d_{\mathbb{R}^{n}})$-continuous. The first statement follows immediately from the continuity of $\clr$ or $\ilr$, respectively. As for the latter case, observe that it is well known \citep[prop. 4]{gao2017properties} that $\sfm\colon\mathbb{R}^{n}\to\mathbb{R}^{n}$ is $1$-Lipschitz continuous. Hence, for all $p,q\in\Delta$ we have
\begin{equation}\label{eq:normequi}
\|p-q\|_{\mathbb{R}^{n}}\le\|\clr(p)-\clr(q)\|_{A} =d_{A}(p,q),
\end{equation}
which yields the second claim.
\end{proof}
Notice that, whereas $d_{A}$ and $d_{\mathbb{R}^{n}}$ are topologically equivalent, they are not strongly equivalent. Indeed, since $d_{\mathbb{R}^{n}}$ is uniformly bounded on $\Delta$ by $\sqrt{2}$ we cannot find some $c>0$ with
\begin{equation*}
d_{A}(p,q)\le cd_{\mathbb{R}^{n}}(p,q),\quad p,q\in\Delta.
\end{equation*}

Finally, let us make a few words on integration and differentiation on $\Delta$. Since the Aitchison simplex is in particular an Abelian group it comes along with a natural reference measure $\lambda_{A}$, which is the Haar measure on $(\Delta,\oplus)$. In the CoDa community it is referred to as the \textit{Aitchison measure} \cite[c.f.][]{pawlowsky2003statistical}, which can be characterized as push-forward of the Lebesgue measure $\lambda_{n-1}$ on $\mathbb{R}^{n-1}$ under $\ilr^{-1}$ or up to multiplicative constants equivalently, as the push forward of $\lambda_{n}$ under $\sfm$. The joint distribution of the first $(n-1)$ marginals of $\lambda_{A}$ (which by slight abuse of notation we call also $\lambda_{A}$) is absolutely continuous with respect to $\lambda_{n-1}$ with Radon-Nikodym derivative
 
\begin{equation*}    
\frac{d\lambda_{A}}{d\lambda_{n-1}}(p_{1},\dots,p_{n-1})=\left(\prod_{i=1}^{n}p_{i}\right)^{-1},
\end{equation*}
where $p_{n}:=1-\sum_{i=1}^{n-1}p_{i}$.\par

As a Hilbert space, the Aitchison simplex also exhibits a natural differential calculus, which, as we will see, is closely related to the classical Fisher information geometry of the simplex (see also \cite{erb2020informationgeometric}). Let us briefly recall the latter.\par 

Set $|x|=\sum x_{i}$ and consider the positive orthant ${\mathbb{R}^{n}_{>0}=\lbrace x\in\mathbb{R}^{n}\ |\ x_{i}> 0\rbrace}$ equipped with the Riemannian metric
\begin{equation*}
 g_{x}(u,v)=\sum_{i=1}^{n}|x|\frac{u_{i}v_{i}}{x_{i}}.
\end{equation*} 
As a submanifold of $(\mathbb{R}^{n}_{>0},g)$, the simplex $\Delta$ inherits a Riemannian structure with inverse metric tensor
\begin{equation}\label{eq:invtens}
(g^{-1}(p))_{ij}:=g^{ij}(p):=p_{i}(\delta_{ij}-p_{j}),\quad p\in\Delta
\end{equation}
and for sufficiently smooth functions $\phi\colon\Delta\to\mathbb{R}$ the gradient at $p\in\Delta$ is given by
\begin{equation*}
\nabla^{g}\phi(p):=g^{-1}(p)\nabla\phi(p)=\sum_{i=1}^{n}p_{i}\left(\partial_{i}\phi(p)-\sum_{j=1}^{n}p_{j}\partial_{j} \phi(p)\right)\epsilon_{i}
\end{equation*}
which is typically referred to as Fisher -or Shahshahani gradient in information geometry and mathematical biology, respectively \citep[c.f.][]{hofrichter2017information,ay2017information,harper2009information,bla}.\par

Now denote by $d^{A}f$ the Fr\'{e}chet derivative on the Aitchison simplex determined as usual via
\begin{equation}\label{eqn: frechet}
\lim_{q\to e}\|q\|^{-1}_{A}|f(p\oplus q)-f(p)-d^{A}f(p)q|=0,
\end{equation}
provided the limit exists (recall that in the previous formula $e$ is the neutral element of the Aitchison simplex $e=n^{-1}\mathbf{1}$). Given $f\in C^{1}(\Delta)$ and $p\in\Delta$, by Riesz representation theorem, we can introduce the \textit{Aitchison gradient} $\nabla^{A}f(p)$ as the unique element in $\Delta$ obeying 
\begin{equation*}
d^{A}f(p)q=\langle\nabla^{A}f(p),q\rangle_{A}\quad\text{ for all } q\in\Delta,
\end{equation*}
 and we claim

\begin{lem}
Let $f\in C^{1}(\Delta)$. Then
\begin{equation*}
\nabla^{A}f=\sfm(\nabla^{g}f).
\end{equation*}
\end{lem} 

\begin{proof}
We set $\bar{f}=f\circ\sfm$ and rewrite $f(p\oplus q)=\bar{f}(\clr(p)+\clr(q))$. Now a Taylor expansion around $\clr(p)$ yields
\begin{equation*}
f(p\oplus q)=f(p)+(\nabla\bar{f}(\clr(p)),\clr(q))+o\left(\|q\|_{A}^{2}\right).
\end{equation*}

But since $\|q\|^{-1}_{A}o(\|q\|_{A}^{2})\rightarrow 0$ as $q\to e$, plugging the previous expansion into (\ref{eqn: frechet}) necessitates
\begin{equation*}
d^{A}f(p)q=(\nabla\bar{f}(\clr(p)),\clr(q)).
\end{equation*}
Next, observe that the Jacobian of $\sfm$ has the remarkable form 
\begin{equation}\label{eq:jac}
D\sfm=\operatorname{diag}(\sfm)-\sfm\otimes\sfm,
\end{equation}
so that by the chain rule
\begin{equation*}
(\nabla \bar{f})\circ\clr(p)=\nabla^{g} f(p).
\end{equation*}
Hence,
\begin{equation*}
(\nabla\bar{f}(\clr(p)),\clr(q))=(\nabla^{g}f(p),\clr(q))_{H}=\langle\sfm(\nabla^{g}f(p)),q\rangle_{A},
\end{equation*}
which yields the claim.   
\end{proof} 
We remark that similar computations have been done before in \citep{simpcalc}. Due to a different definition for the derivative their gradient differs from ours and simply equals $\nabla^{g}$.\par

\invisiblesection{The Aitchison diffusion aka Brownian motion on the simplex}     
  
\medskip      
 \begin{center}
3. T{\footnotesize HE AITCHISON DIFFUSION AKA BROWNIAN MOTION ON THE SIMPLEX}
\end{center}

Recall that, given a probability space $(\Omega,\mathcal{F},\mathbb{P})$ and a topological Abelian group $(G,\ast)$, a $G$-valued stochastic process is called \textit{Brownian motion on $G$} provided
   
\begin{enumerate}[label=A\arabic*.]
\item For $0\le t_{1}\le t_{2}\le\dots\le t_{m}\le T$ and every $m\in\mathbb{N}$, the increments \par 
\noindent ${X_{t_{1}},X_{t_{2}}\ast X_{t_{1}}^{-1},\dots, X_{t_{m}}\ast X^{-1}_{t_{m-1}}}$ are mutually independent. 
\item For any $0\le s\le t\le T$ the law of the increments $X_{t}\ast X_{s}^{-1}$ depends only on $t-s$.
\item $t\mapsto X_{t}$ is continuous a.s.
\end{enumerate}

The main result of this section establishes a deep connection between stochastic replicator dynamics and Brownian motion:

\begin{thm}[Aitchison diffusion as Brownian motion on $\Delta$]\label{thm: Aidiff}
For every $T>0$ there exists a unique $\Delta$-valued process $X$ solving on $[0,T]$ the Stratonovich SDE

\begin{equation}\tag{\ref{eq:aidif}}
\begin{aligned}
dX^{i}_{t}&=X^{i}_{t}\left(\circ dB^{i}_{t}-\sum_{j=1}^{n}X^{j}_{t}\circ dB^{j}_{t}\right),\quad i=1,\dots n\\
X_{0}&=p\in\Delta
\end{aligned}
\end{equation} 
Moreover, $X$ satisfies the properties A1.-A3. with $G=(\Delta,\oplus)$ the Aitchison simplex. Thus, $X$ is a Brownian motion on $\Delta$.
\end{thm} 

We approach Theorem \ref{thm: Aidiff} by some preliminary considerations. First observe that since $(\Delta,\oplus,\odot)$ is a finite dimensional vector space, there is a canonical way to introduce Brownian motion on the Aitchison simplex. Namely, we simply take an orthonormal basis $\lbrace e_{1},\dots,e_{n-1}\rbrace$ of $\Delta$ and $n-1$ independent (one-dimensional) standard Brownian motions $B^{i},\ i=1,\dots,n-1$ and define for all $t\ge0$
\begin{equation}\label{eq:BMAicoo}
\hat{X}_{t}:=\bigoplus_{i=1}^{n-1}B^{i}_{t}\odot e_{i}.
\end{equation}
Then clearly $\hat{X}$ satisfies A1.-A3. Next, we would like to find a characterisation of $\hat{X}$ in conventional Euclidean coordinates. Of course, (\ref{eq:BMAicoo}) means nothing but $\hat{X}=\ilr^{-1}(B)$. Thus, by the Stratonovich chain rule we find that $\hat{X}$ satisfies
\begin{align*}
d\hat{X}^{i}_{t}&=d\sfm_{i}(\Psi^{\top}B_{t})=\sum_{j=1}^{n-1}\sum_{k=1}^{n}\Psi_{jk}(\partial_{k}\sfm_{i})(\Psi^{\top}B_{t})\circ dB^{j}_{t}\\
&=\sum_{j=1}^{n-1}\sum_{k=1}^{n}\Psi_{jk}\sfm_{i}(\Psi^{\top}B_{t})(\delta_{ik}-\sfm_{k}(\Psi^{\top}B_{t}))\circ dB^{j}_{t}\\
&=\sum_{j=1}^{n-1}\sum_{k=1}^{n}\Psi_{jk}g^{ik}(\hat{X}_{t})\circ dB^{j}_{t},
\end{align*}
with $g^{-1}$ defined as in (\ref{eq:invtens}). If we introduce for $i=1,\dots,n-1$ the maps $\hat{Z}_{i}\colon\Delta\to\mathbb{R}^{n}$ with
\begin{equation}\label{eq:vecmap}
\hat{Z}_{i}(p):=\sum_{j,k=1}^{n}\Psi_{ik}g^{jk}(p)\epsilon_{j},
\end{equation}  
then the Brownian motion on $\Delta$ defined as in (\ref{eq:BMAicoo}) satisfies
\begin{equation*}
d\hat{X}_{t}=\sum_{i=1}^{n-1}\hat{Z}_{i}(\hat{X}_{t})\circ dB^{i}_{t}.
\end{equation*} 
Therefore, the generator $\hat{L}$ of $\hat{X}$ is canonically given in H\"ormander form as
\begin{equation*}
\hat{L}=\frac{1}{2}\sum_{i=1}^{n-1}\hat{Z}_{i}^{2},
\end{equation*} 
where, by slight abuse of notation, we identified with the maps in (\ref{eq:vecmap}) the vector fields on $\Delta$ given by
\begin{equation*}
\hat{Z}_{i}(p):=\sum_{j,k=1}^{n}\Psi_{ik}g^{jk}(p)\partial_{j}.
\end{equation*}  
Now setting
\begin{equation*}
g^{ik}_{l}(p):=\partial_{l}g^{ik}(p)=\delta_{il}\delta_{ik}-p_{k}\delta_{il}-p_{i}\delta_{lk},
\end{equation*}
and
\begin{equation*}
G^{ij}(p):=(G^{-1}(p))_{ij}:=\left({g^{-1}(p)}^{2}\right)_{ij}=p_{i}\left(p_{i}\delta_{ij}-p_{i}p_{j}-p_{j}^{2}+p_{j}\sum_{k=1}^{n}p_{k}^{2}\right)
\end{equation*}
we can expand $\hat{L}$ and obtain, using (\ref{eq:contrastprop}) and the the fact that the rows of $g^{-1}$ sum to zero,
\begin{align}\label{eq:ilrgen}
2\hat{L}&=\sum_{j,k,l,m=1}^{n}(\Psi^{\top}\Psi)_{km}\left(g^{jk}g^{lm}\partial_{jl}+g^{jk}g^{lm}_{j}\partial_{l}\right)\nonumber\\
&=\sum_{j,l=1}^{n}G^{jl}\partial_{jl}+\sum_{l=1}^{n}\left(\sum_{j,m=1}^{n}g^{jm}g^{lm}_{j}\right)\partial_{l}.
\end{align}
From this observation we can easily infer the
  
\begin{proof}[Proof of Theorem \ref{thm: Aidiff}]
Although existence and uniqueness of solutions to (\ref{eq:aidif}) is well known, since (\ref{eq:aidif}) is just a special case of the general stochastic replicator equation (\ref{eq:fudsde}), we briefly sketch the argument for the sake of completeness. Recall, that the Aitchison diffusion (\ref{eq:aidif}) in It\={o} form obeys
\begin{equation}\label{eq:aidiffito}
dX_{t}=b(X_{t})dt+g^{-1}(X_{t})dB_{t},
\end{equation}
with Stratonovich corrector
\begin{equation}\label{eq:stratcor}
b_{i}(p)=\frac{1}{2}\sum_{j,k=1}^{n}g^{jk}(p)g^{ik}_{j}(p)=p_{i}\left(\sum_{k=1}^{n}p^{2}_{k}-p_{i}\right)=-\left(g^{-1}(p)p\right)_{i},\quad i=1,\dots,n.
\end{equation}
Observe that both maps, $g^{-1}\colon\Delta\to\mathbb{R}^{n\times n}$ and $b\colon\Delta\to\mathbb{R}^{n}$ are Lipschitz-continuous (w.r.t the standard topology on $\Delta$). Indeed, for every $p,q\in\Delta$ we have, 
\begin{align*}
&\|g^{-1}(p)-g^{-1}(q)\|^{2}=\sum_{i,j=1}^{n}(g^{ij}(p)-g^{ij}(q))^{2}=\sum_{i,j=1}^{n}((p_{i}-q_{i})\delta_{ij}+q_{i}q_{j}-p_{i}p_{j})^{2}\\
\le& 2\|p-q\|^{2}+2\sum_{i,j=1}^{n}(q_{i}q_{j}-q_{i}p_{j}+q_{i}p_{j}-p_{i}p_{j})^{2}\le10\|p-q\|^{2}.
\end{align*}
Regarding the drift component, one finds

\begin{align*}
&\|b(p)-b(q)\|^{2}\le 2\sum_{i=1}^{n}\left(q_{i}^{2}-p_{i}^{2}\right)^{2}+2\sum_{i=1}^{n}\left(p_{i}\sum_{k=1}^{n}p_{k}^{2}-q_{i}\sum_{k=1}^{n}q_{k}^{2}\right)^{2}\\
\le & 8\|p-q\|^{2}+4\sum_{i=1}^{n}p_{i}^{2}\left(\sum_{k=1}^{n}p_{k}^{2}-q_{k}^{2}\right)^{2}+4\sum_{i=1}^{n}(p_{i}-q_{i})^{2}\left(\sum_{k=1}^{n}q_{k}^{2}\right)^{2}\\
\le & (12+16n)\|p-q\|^{2}.
\end{align*}

Since moreover, $\operatorname{im}g^{-1}(p)=T_{p}\Delta=H$ (with $H$ as defined in (\ref{eq:H})) holds for every $p\in\Delta$, existence and uniqueness of a continuous and $\Delta$-valued solution to (\ref{eq:aidif}) follow by standard Picard iteration. In particular such a solution satisfies A3.\par 

We are left to check that $X$ satisfies the properties A1. and A2. But comparing the Stratonovich corrector (\ref{eq:stratcor}) with the first order part of $\hat{L}$ in (\ref{eq:ilrgen}), we realise that $\hat{L}$ and the generator $L$ associated to the Aitchison diffusion $X$ from (\ref{eq:aidif}) and (\ref{eq:aidiffito}), respectively, coincide. Consequently, the Brownian motion on the Aitchison simplex $\hat{X}$ as introduced in (\ref{eq:BMAicoo}) and the Aitchison diffusion $X$ have the same law and thus $X$ also satisfies A1. and A2.
\end{proof}

Subsequently, we want to point out a different and instructive way to deduce
 Theorem \ref{thm: Aidiff} by rather geometric arguments. At that, we mainly follow the lines of \citep[][ch. 8]{stroock2003markov}. First, observe that akin to $\hat{X}$, we can rewrite (\ref{eq:aidif}) as

\begin{equation}\label{eqn:aistraform}
dX_{t}=\sum_{i=1}^{n}Z_{i}(X_{t})\circ dB^{i}_{t},\quad X_{0}=p\in\Delta,
\end{equation}
where now
\begin{equation*}
 Z_{i}:=\sum_{j=1}^{n}g^{i j}\epsilon_{j}.
 \end{equation*}
So, the generator $L$ of the Aitchison diffusion in H\"ormander form reads
\begin{equation}\label{eq:L}
L=\frac{1}{2}\sum_{i=1}^{n}Z_{i}^{2},
\end{equation}
where again we identified the maps $Z_{i}$ with the corresponding vector fields
\begin{equation}\label{egn:vecfields}
 Z_{i}:=\sum_{j=1}^{n}g^{i j}\partial_{j}.
\end{equation}
We omit the proof of the following result, which is tedious but straight forward. 
 
\begin{lem}\label{lem:com}
The vector fields $Z_{1},\dots,Z_{n}$ as defined in (\ref{egn:vecfields}) commute.
\end{lem}
  
Elementary facts from ODE theory ensure that we can determine uniquely a smooth map ${E\colon\mathbb{R}^{n}\times\Delta\to\mathbb{R}^{n}}$ satisfying  
\begin{equation}\label{eq:vecsys}
\frac{d}{dt}E(t\xi,p)=\sum_{i=1}^{n}\xi_{i}Z_{i}(E(t\xi,p)),\quad E(0,p)=p\in\Delta.
\end{equation} 
Moreover, Lemma \ref{lem:com} necessitates that
\begin{equation}\label{eq:vecprop}
\partial_{\xi_{k}}E(\xi,p)=Z_{k}(E(\xi,p))
\end{equation} 
for all $(\xi,p)\in\mathbb{R}^{n}\times\Delta$ and $k=1,\dots,n$. Now take an $n$-dimensional standard Brownian motion $B$ be on $(\Omega,\mathcal{F},\mathcal{F}_{t},\mathbb{P})$ and consider $X:=E(B,p)$.
Then, the Stratonovich chain rule and (\ref{eq:vecprop}) entail
\begin{equation*}
dX_{t}=\sum_{i=1}^{n}Z_{i}(X_{t})\circ dB^{i}_{t}
\end{equation*}  
and $X_{0}=E(0,p)=p\in\Delta$. Thus, $X$ is a solution to (\ref{eq:aidif}) and moreover defines a flow of diffeomorphisms on $\Delta$. On the other hand, recall from (\ref{eq:jac}) that
\begin{equation*}
D\sfm=\operatorname{diag}(\sfm)-\sfm\otimes\sfm.
\end{equation*}
Using this identity, we find that $E(\xi,p)=\sfm(\xi)\oplus p$ provides a solution to (\ref{eq:vecsys}), which in turn implies that the Aitchison diffusion starting in $p\in\Delta$ is simply given by
\begin{equation}
X_{t}=p\oplus\sfm(B_{t}).
\end{equation}
Of course, $X$ inherits the properties A1.- A3. from $B$ by the transformation rules (\ref{eqn: traforule}). A simulation of a trajectory of the Aitchison diffusion on $\Delta_{3}$ is depicted as a ternary plot in figure 1.

\begin{figure}[t]
	\centering
  \includegraphics[trim = 5mm 40mm 5mm 40mm, clip, width=0.65\textwidth,]{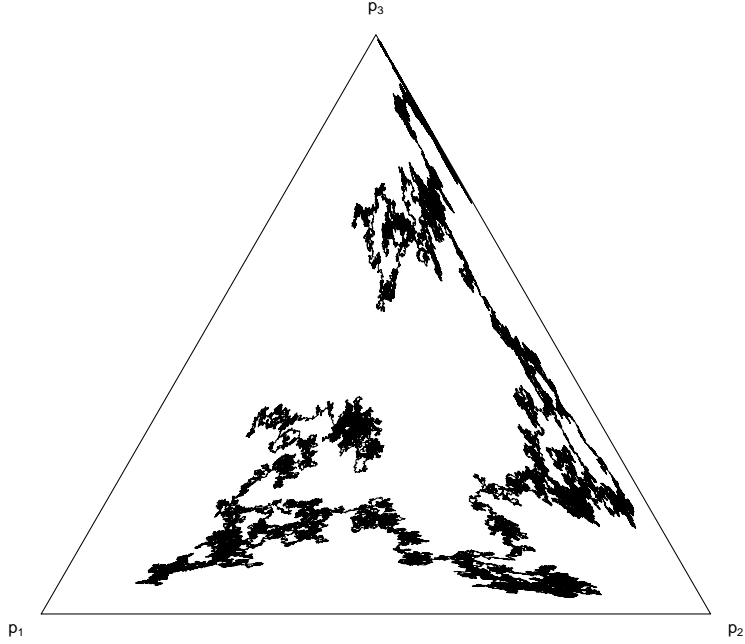}
	\caption{Simulation of a trajectory of an Aitchison diffusion starting at the barycenter}
	\label{fig2}
\end{figure}

In the following we denote by $P_{t}^{X}$ the Markov semigroup associated to $X$, i.e. for bounded and measurable functions $f$ on $\Delta$ we set
\begin{equation}
P^{X}_{t}f(p):=\mathbb{E}\left[f(X_{t})|X_{0}=p\right].
\end{equation}
Also, we write $P_{t}^{B}$ for the Brownian semigroup. The characterisations of the Aitchison diffusion which we discussed so far imply immediately the following properties of $P_{t}^{X}$.

\begin{cor}[invariant measure and density kernel of the Aitchison semigroup]
\hfill
\begin{enumerate}[label=(\roman*)]
\item The Aitchison measure $\lambda_{A}$ is invariant and reversible for $P_{t}^{X}$.
\item $P_{t}^{X}$ admits a density kernel $(0,\infty)\times\Delta\times\Delta\ni (t,p,q)\mapsto p_{t}(p,q)$ with respect to $\lambda_{A}$, which is given by
\begin{equation*}
p_{t}(p,q)=(2\pi t)^{-\frac{n-1}{2}}e^{-\frac{\|p\ominus q\|_{A}^{2}}{2t}}.
\end{equation*}
In particular, $d_{A}$ is the intrinsic metric for $X$ and we have the Varadhan short time asymptotics
\begin{equation}\label{eq:varadhan}
\lim_{t\downarrow0}t\ln p_{t}(p,q)=-\frac{d^{2}_{A}(p,q)}{2}.
\end{equation}

\end{enumerate}
\end{cor}

\begin{proof}
$(i)$ Let $f\colon\Delta\to\mathbb{R}$ be positive, measurable and bounded. Then,
\begin{align*}
\int_{\Delta}P_{t}^{X}f(p)\lambda_{A}(dp)&=\int\mathbb{E}\left[\hat{X}_{t}|\hat{X}_{0}=\ilr^{-1}(x)\right]dx=\int P^{B}_{t}f\circ\ilr^{-1}(x)dx\\
&=\int f\circ\ilr^{-1}(x)dx=\int_{\Delta}f(p)\lambda_{A}(dp)
\end{align*}  
proving invariance of $\lambda_{A}$. Reversibility follows by a similar argument.\par 
For $(ii)$, let $p:=\ilr^{-1}(x)$ for some $x\in\mathbb{R}^{n-1}$ and observe that
\begin{align*}
P^{X}_{t}f(p)&=\mathbb{E}\left[f\circ\ilr^{-1}(B_{t})|B_{0}=x\right]=\int f\circ\ilr^{-1}(y)(2\pi t)^{-\frac{n-1}{2}}e^{-\frac{\|x-y\|^{2}}{2t}}dy\\
&=\int_{\Delta}f(q)p_{t}(p,q)\lambda_{A}(dq),
\end{align*}
with
\begin{equation*}
p_{t}(p,q)=(2\pi t)^{-\frac{n-1}{2}}e^{-\frac{\|p\ominus q\|_{A}^{2}}{2t}}.
\end{equation*}
The previous identity immediately gives (\ref{eq:varadhan}).
\end{proof}

Notice that $\lambda_{A}$ is not a finite measure on $\Delta$ in accordance with the fact that $X_{t}$ is transient, meaning $X_{t}\rightarrow\partial\Delta$ as $t\rightarrow\infty$ a.s. This again follows from the simple observation
\begin{equation*}
\lim_{t\to\infty}\|X_{t}\|_{A}=\lim_{t\to\infty}\|B_{t}\|=\infty\ \text{a.s.}
\end{equation*}
In fact, one can even identify the distribution of the limit. Namely, first observe that
\begin{equation*}
\lim_{t\to\infty}X_{t}=\lim_{t\to\infty}p\oplus\sfm(B_{t})\overset{d}{=}p\oplus \lim_{t\to\infty}\sfm(\sqrt{t}B_{1})=p\oplus\argmax(B_{1}).
\end{equation*}
Since that standard normal distribution attributes zero mass to the set of points in $\mathbb{R}^{n}$ which have no distinct maximum it follow that $\lim_{t\to\infty}X_{t}\sim \operatorname{Uniform}(\lbrace\epsilon_{1},\dots,\epsilon_{n}\rbrace)$.

\invisiblesection{Three approximation results for the Aitchison diffusion}  

\medskip      
 \begin{center}
4. T{\footnotesize HREE APPROXIMATION RESULTS FOR THE AITCHISON DIFFUSION}
\end{center}

In this section we want to provide three approximation results for the Aitchison diffusion, which rely on classical theorems from stochastic analysis or optimal transport and PDE theory, namely Wong-Zakai approximation, Donsker's invariance principle and the JKO-scheme for the heat equation.\par\smallskip

Recall, that in the introduction we alleged that the Aitchison diffusion, as described via the Stratonovich SDE
\begin{equation*}
\begin{aligned}
dX^{i}_{t}&=X^{i}_{t}\left(\circ dB^{i}_{t}-\sum_{j=1}^{n}X^{j}_{t}\circ dB^{j}_{t}\right),\quad i=1,\dots n\\
X_{0}&=p\in\Delta,
\end{aligned}
\end{equation*}
has the natural interpretation of being a replicator equation in a white noise fitness landscape. The first statement gives a justification for this assertion. Indeed, we will see that (\ref{eq:aidif}) can be derived from a replicator dynamic in a coloured noise fitness landscape, when sending the correlation length to zero.\par 

To this end, consider on a probability space $(\Omega,\mathcal{F},\mathbb{P})$ the fitness landscape $y=(y_{t})_{t\ge 0}$, which is independent of the current population state, but evolves randomly and continuously in time as a Gaussian process with correlation structure
\begin{equation*}\label{eq: OUcorr}
\mathbb{E}[y^{i}_{s}y^{j}_{t}]=\delta_{ij}e^{\lambda^{-2}|s-t|},\quad i,j=1,\dots n,
\end{equation*}
where $\lambda>0$ is a correlation length parameter. In other words, we assume that the fitness landscape $y$ is given by the $n$-dimensional Ornstein-Uhlenbeck process 
\begin{equation*}\label{eq:OUSDE} 
\lambda^{2}dy_{t}=-y_{t}dt+\lambda dB_{t},
\end{equation*}
with $B$ an $n$-dimensional Brownian motion and $y_{0}\sim\mathcal{N}(0,\operatorname{Id}_{\mathbb{R}^{n}})$. Then we ascertain the following

\begin{thm}[Approximation of the Aitchison diffusion by replicator dynamics in a Gaussian fitness landscape]
For every $\lambda>0$ there exists a unique solution $p^{\lambda}$ to 

\begin{equation}\label{eq:replgaus}
\begin{aligned}
\dot{p}_{i}(t)&=\lambda^{-1}p_{i}(t)\left(y^{i}_{t}-\bar{y}_{t}\right),\quad i=1,\dots,n\\
p(0)&=p\in\Delta, 
\end{aligned}
\end{equation}  
where $\bar{y}_{t}=\sum_{1}^{n}p_{i}(t)y^{i}_{t}$ is the mean fitness. Moreover, the solutions $p^{\lambda}$ converge weakly to the Aitchison diffusion $X$ as $\lambda$ goes to zero. 
\end{thm}

\begin{proof}
The existence of a unique solution $p^{\lambda}$ to (\ref{eq:replgaus}) follows again by standard Picard iteration. The convergence result is a classical incidence of the Wong-Zakai approximation \citep[c.f.][]{eugene1965relation}. Weak convergence can be proven along the lines of \citep[sec. 5.1.]{pavliotis2014stochastic}. Invoking arguments from rough path theory one could obtain stronger convergence results in H\"older topologies as well \citep[c.f.][]{friz2014course, lyons2007differential}.
\end{proof}

Our second approximation result establishes a connection between random walks on the Aitchison simplex and (\ref{eq:aidif}) in the spirit of Donsker's invariance principle \citep[c.f.][]{newman1981invariance}. Therefore, first recall that the time discrete analogue of the replication dynamic (\ref{eq: rep}) is given by the dynamical system
   \begin{equation*}
\begin{aligned}
 p_{i}^{(k+1)}&=\frac{p_{i}^{(k)}f^{i}\left(p^{(k)}\right)}{\sum_{j=1}^{n}p_{j}^{(k)}f^{j}\left(p^{(k)}\right)}\\
 p^{(0)}&=p\in\Delta.
\end{aligned}   
   \end{equation*}
It was pointed out in \citep{harper2009information} and \citep{shalizi2009dynamics} that we may think of this dynamic as modelling species adaptation by means of generation-wise Bayesian updating. \par  
As previously, we replace the frequency dependent fitness $f$ by random entities. More precisely, on some probability space $(\Omega,\mathcal{F},\mathbb{P})$ we take iid $\Delta$-valued random variables $f_{(1)},f_{(2)},\cdots$  and consider
   \begin{equation*}\label{eqn: simpwalk}
\begin{aligned}
 p_{i}^{(k+1)}&=\frac{p_{i}^{(k)}f_{(k+1)}^{i}}{\sum_{j=1}^{n}p_{j}^{(k)}f_{(k+1)}^{j}}\\
 p^{(0)}&=p\in\Delta.
\end{aligned}   
   \end{equation*}
Using the Aitchison calculus we rewrite the previous updating rule simply as
 
\begin{equation}\label{eqn: atchiwalk}
\begin{aligned}
 p^{(k+1)}&=p^{(k)}\oplus f_{(k+1)}\\
 p^{(0)}&=p\in\Delta.
\end{aligned}  
\end{equation}
    or explicitly
    \begin{equation*}
    p^{(k)}=\bigoplus_{i=1}^{k}f_{(i)}\oplus p.
    \end{equation*}
That means $\left(p^{(k)}\right)_{k\ge0}$ is nothing but the random walk on the Aitchison simplex $(\Delta,\oplus,\odot)$, induced through the iid random variables $\left(f_{(k)}\right)_{k\ge 1}$. As a direct consequence of the continuous mapping theorem and Donsker's invariance principle we infer

\begin{thm}[Approximation of the Aitchison diffusion by random walks]
Let $f_{(1)},f_{(2)},\dots$ be iid random variables with values in $\Delta$ such that for all $k\in\mathbb{N}$
\begin{equation*}
\mathbb{E}\clr(f_{(k)})=0
\end{equation*}
and
\begin{equation*}
\operatorname{cov}\left(\ilr_{i}(f_{(k)}),\ilr_{j}(f_{(k)})\right)=\delta_{ij},\quad i,j=1,\dots,n-1.
\end{equation*}
Let $(p^{(k)})_{k\ge0}$ be a random walk on $\Delta$ as given in (\ref{eqn: atchiwalk}) and assume for simplicity $p^{(0)}=p=e$, with $e$ being the barycenter of $\Delta$. Define the linear interpolation of the random walk as
\begin{equation*}
p_{t}:=p^{(\lfloor t\rfloor)}\oplus(t-\lfloor t\rfloor)\odot p^{(\lfloor t \rfloor +1)}.
\end{equation*}
Then the family of $C([0,1],\Delta)$- valued random elements $(n^{-\frac{1}{2}}\odot p_{nt})_{t\in[0,1]},\ n\in\mathbb{N},$ converges weakly to the Aichtison diffusion $(X_{t})_{t\in[0,1]}$ starting in $X_{0}=e$ as $n\to\infty$.
\end{thm}
  
Whereas the first two approximation results are in essence probabilistic, the last theorem in this section is based on a gradient flow interpretation of the Fokker-Planck equation associated to (\ref{eq:aidif}) and relies on the seminal work of Otto et al. in \citep{jordan1998variational} and \citep{otto2001geometry}.\par 

Let us start with a few definitions. We denote by $\mathcal{P}(\Delta)$ the set of all probability measures on $\Delta$ and by
\begin{equation*}
\mathcal{P}_{2}(\Delta):=\left\lbrace\mu\in\mathcal{P}(\Delta)\ \colon \int_{\Delta}\|p\|_{A}^{2}\mu(dp)<\infty\right\rbrace.
\end{equation*}
Then a natural distance measure on $\mathcal{P}_{2}(\Delta)$ is the Wasserstein distance
\begin{equation}\label{eq:Wadi}
W^{2}_{A}(\mu,\nu):=\inf_{\pi\in\Pi(\mu,\nu)}\int_{\Delta\times \Delta}d_{A}^{2}(p,q)\pi(dpdq),
\end{equation}
where for $\mu,\nu\in\mathcal{P}_{2}(\Delta)$ we have
\begin{equation*}
\Pi(\mu,\nu):=\lbrace\pi\in\mathcal{P}(\Delta\times \Delta) \, \colon \, \pi(A\times \Delta)=\mu(A)  \text{ and }  \pi(\Delta\times B)=\nu(B),\ A,B\in \mathcal{B}(\Delta) \rbrace.
\end{equation*}
Notice, that the Wasserstein distance on $\mathcal{P}_{2}(\Delta)$ is linked to the usual Wasserstein metric on $\mathcal{P}_{2}(\mathbb{R}^{n-1})$ through
\begin{equation}\label{eq:waequi}
W^{2}_{A}(\mu,\nu)=\inf_{\pi\in\Pi(\ilr_{\#}\mu,\ilr_{\#}\nu)}\int_{\mathbb{R}^{n-1}}\|x-y\|^{2}\pi(dxdy)=:W^{2}_{\mathbb{R}^{n-1}}(\ilr_{\#}\mu,\ilr_{\#}\nu).
\end{equation}
Finally, we denote by $\mathcal{S}_{A}$ the Boltzmann entropy on the Aitchison simplex:
\begin{equation*}
\mathcal{S}_{A}(\mu):=\begin{cases}
\int_{\Delta}\ln\left(\frac{d\mu}{d\lambda_{A}}(p)\right) \mu(dp), & \mu <\hspace{-0.2cm}< \lambda_{A}\\
+\infty, & \text{else.}
\end{cases}
\end{equation*}
Now let $X$ be an Aitchison diffusion with $X_{0}\sim\mu_{0}$ for some $\mu_{0}\in\mathcal{P}_{2}(\Delta)$. We denote by $\mu_{t}(dp)=\mu_{t}(p)\lambda_{A}(dp)$ the law of $X_{t}$. Then the densities $(\mu_{t})_{t\ge 0}$ provide a solution to the heat equation on the Aitchison simplex

\begin{equation}\label{eq:Aiheat}
\partial_{t}\mu=L\mu,\quad\mu_{|t=0}=\mu_{0}
\end{equation} 
and satisfy for any $\nu\in\mathcal{P}_{2}(\Delta)$ with finite entropy
\begin{equation}\label{eq:evi}
\frac{d}{dt}W^{2}_{A}(\mu_{t},\nu)\le\mathcal{S}_{A}(\nu)-\mathcal{S}_{A}(\mu_{t}).
\end{equation}
The meaning of the previous evolution-variational inequality is that (the laws of) the Aitchison diffusion evolve as a Wasserstein gradient flow of the Boltzmann entropy. To see, why (\ref{eq:evi}) is true, let $\rho$ be a solution to the heat equation on $\mathbb{R}^{n-1}$. Then, denoting by
\begin{equation*}
\mathcal{S}_{\mathbb{R}^{n-1}}(\mu):=\begin{cases}
\int_{\mathbb{R}^{n-1}}\ln\left(\frac{d\mu}{d\lambda_{\mathbb{R}^{n-1}}}(x)\right) \mu(dx), & \mu <\hspace{-0.2cm}< \lambda_{\mathbb{R}^{n-1}}\\
+\infty, & \text{else}
\end{cases}
\end{equation*}
the usual entropy on $\mathbb{R}^{n-1}$, we know that
\begin{equation*}
\frac{d}{dt}W^{2}_{\mathbb{R}^{n-1}}(\rho_{t},\eta)\le\mathcal{S}_{\mathbb{R}^{n-1}}(\eta)-\mathcal{S}_{\mathbb{R}^{n-1}}(\rho_{t}),
\end{equation*}
for all probability measures $\eta$ on $\mathbb{R}^{n-1}$ with finite second moment and entropy \citep[c.f.][]{erbar2010, ambrosio2008gradient}. Next, notice that if $\nu\in\mathcal{P}(\Delta)$ has finite entropy, then

\begin{align*}
\mathcal{S}_{\mathbb{R}^{n-1}}(\ilr_{\#}\nu)&=\int_{\mathbb{R}^{n-1}}\ln\left(\frac{d\ilr_{\#}\nu}{d\lambda_{\mathbb{R}^{n-1}}}(x)\right) \ilr_{\#}\nu(dx)=\int_{\Delta}\ln\left(\frac{d\ilr_{\#}\nu}{d\ilr_{\#}\lambda_{A}}(\ilr(p))\right) \nu(dp)=\mathcal{S}_{A}(\nu).
\end{align*}
Now we easily infer (\ref{eq:evi}) from the previous identity, (\ref{eq:waequi}) and the fact that $\mu_{t}=\ilr^{-1}_{\#}\rho_{t}$.\par 
The observation that the Aitchison diffusion can be interpreted as a Wasserstein gradient flow now provides an immediate way to approximate the laws $\mu_{t}$ be means of a steepest descent algorithm, in perfect analogy to the JKO-scheme invented in \citep{jordan1998variational}.

\begin{thm}[Approximation of the Aitchsion diffusion by JKO-scheme]
For every $t>0$ we denote by
\begin{equation*}
J_{t}(\bar{\mu}|\mu_{0}):=\mathcal{S}(\bar{\mu})-\mathcal{S}(\mu_{0})-\frac{1}{2t}W_{A}^{2}(\bar{\mu},\mu_{0})\quad \text{ and } \quad K_{t}[\rho_{0}]:=\argmin_{\bar{\mu}\in\mathcal{P}_{2}(\Delta)}J_{t}(\bar{\mu}|\mu_{0})
\end{equation*}
Then,
\begin{equation*}
\mu_{t}:=\lim_{n\to\infty}(K_{t/n})^{n}[\mu_{0}],\quad t>0
\end{equation*}
is the law of the Aitchison diffusion (at time $t$) and its $\lambda_{A}$-density solves the heat equation (\ref{eq:Aiheat}).
\end{thm}

\invisiblesection{Drift diffusion on the Aitchison simplex}  

 \begin{center}
5. D{\footnotesize RIFT DIFFUSIONS ON THE AITCHISON SIMPLEX}
\end{center}

In this final section we are concerned with the influence of drift terms on the Aitchison diffusion. More precisely, we are interested in the long time behaviour of Markov processes which are associated to differential operators of the form $L_{Z_{0}}:=L+Z_{0}$, where $L$ is the generator of the Aitchison diffusion as in (\ref{eq:L}) and $Z_{0}$ is a vector field on $\Delta$. Rather then aiming for maximal generality, our focus is on examples which seem interesting for applications in e.g. mathematical biology and game theory, foremost the case
\begin{equation*}
Z_{0}f(p)=Ap\cdot\nabla^{g}f(p),
\end{equation*}
which corresponds to the stochastic replication equation. Our investigations will be split into two parts. First, we derive structural properties of invariant measures for the diffusion processes in question. Afterwards, we provide quantitative statements about the relaxation to equilibrium by means of Wasserstein contraction estimates. For both concerns we will benefit from a frequent change of perspective between the usual Euclidean picture on one hand and the Aitchison geometry on the other hand. \par

As a first example for this approach, consider

\begin{definition}[Definition](ODE on the Aitchison simplex)\textbf{.}
Let $F\colon\Delta\to\Delta$. By a solution to the ordinary differential equation (ODE) on the Aitchison simplex
\begin{equation}\label{eq:AIODE}
\dot{p}(t)=F(p(t)),\quad p(0)=p\in\Delta,
\end{equation}
we mean a continuous map $[0,\infty)\ni t\mapsto p(t)\in\Delta$, obeying

\begin{equation*}
\langle p(t),q\rangle_{A}=\langle p,q\rangle_{A}+\int_{0}^{t}\langle F(p(s)),q\rangle_{A}ds,\quad t\ge 0
\end{equation*}
for all $q\in\Delta$. 
\end{definition}

\begin{lem}\label{lem:bla}
	Fix $p\in\Delta$. Then $(p(t))$ is a solution to (\ref{eq:AIODE}) if and only if it is a solution to the conventional ODE
	\begin{equation}\label{eq:OODE}
	\dot{p}(t)=g^{-1}(p(t))\clr(F(p(s))),\quad p(0)=p.
	\end{equation}
	
\end{lem}

\begin{proof}
	We set $\tilde{p}(t):=\clr(p(t))$. Recall that for $p,q\in\Delta$ we have $\langle p,q\rangle_{A}=(\clr(p),\clr(q))_{H}$. Thus, $(p(t))$ is a solution to (\ref{eq:AIODE}) iff for all $h\in H$
	\begin{equation}\label{eq:w}
	(\tilde{p}(t),h)_{H}=(\tilde{p},h)_{H}+\int_{0}^{t}(\clr(F(p(s))),h)_{H}ds,
	\end{equation}
	Therefore, if $(p(t))$ solves (\ref{eq:AIODE}), it follows
	\begin{equation*}
	\dot {p}(t)=\frac{d}{dt}\sfm(\tilde{p}(t))=g^{-1}(p(t))\dot{\tilde{p}}(t)=g^{-1}(p(t))\clr(F(p(s)))
	\end{equation*}
	which is (\ref{eq:OODE}). On the other hand, if $(p(t))$ solves (\ref{eq:OODE}), then since
	\begin{equation*}
	\frac{d}{dt}\ln p_{i}(t)=\ln F_{i}(p(t))-\sum_{j=1}^{n}p_{j}(t)\ln(F_{j}(t)),
	\end{equation*}
	we find
	\begin{equation*}
	\dot{\tilde{p}}(t)=\frac{d}{dt}\clr(p(t))=\clr(F(p(t)))
	\end{equation*}
	which, using (\ref{eq:w}), yields that $(p(t))$ solves (\ref{eq:AIODE}).
\end{proof}

\begin{remark}
	Notice that the transformation above has a nice biological interpretation. Think of $F$ in (\ref{eq:AIODE}) as Wrightian fitness. Then the flow $(p(t))$ on the Aitchison simplex driven by $F$ corresponds in the usual Euclidean notation to a replicator equation with fitness landscape $(\ln F_{1},\dots,\ln F_{n})$, i.e. the Malthusian fitness \citep[c.f.][]{wu2013interpretations}. 
\end{remark}
   
In particular, given some payoff matrix $A\in\mathbb{R}^{n\times n}$, if we define
\begin{equation}\label{eq:theta}
\theta\colon\Delta\to\Delta,\quad p\mapsto\theta(p):=\sfm(Ap),
\end{equation}
then
\begin{equation*}
\dot{p}=\theta(p),\quad p(0)=p\in\Delta
\end{equation*}
is equivalent to the replicator equation with linear fitness landscape $p\mapsto Ap$:
\begin{equation}\label{eq:requad}
\dot{p}_{i}(t)=p_{i}(t)\left((Ap(t))_{i}-p(t)\cdot Ap(t)\right),\quad i=1,\dots,n
\end{equation}
provided $p(0)=p\in\Delta$.\par 
Let us try to establish a stochastic analogue of the previous observation. 
\begin{definition}[Definition](SDE on the Aitchison simplex)\textbf{.}
	Let $F\colon\Delta\to\Delta$ and $p\in\Delta$. Given a filtered probability space $\mathbf{\Omega}=(\Omega,\mathcal{F},(\mathcal{F})_{t},\mathbb{P})$ and an Aitchison diffusion $X$ on $\mathbf{\Omega}$, we call an $\mathcal{F}_{t}$-adapted, time-continuous and $\Delta$-valued process $Y$ a solution to the SDE on the Aitchison simplex, 
	\begin{equation}\label{eq:Ailang}  
	dY_{t}=F(Y_{t})dt\oplus dX_{t},\quad Y_{0}=p,
	\end{equation}
	provided that $\mathbb{P}$-a.s.
	\begin{equation}\label{eq:ASDEdef}
	\langle Y_{t},q\rangle_{A}=\langle p,q\rangle_{A}+\int_{0}^{t}\langle F(Y_{s}),q\rangle_{A}ds+\langle X_{t},q\rangle_{A},\quad t\ge 0
	\end{equation}
	for all $q\in\Delta$. 
\end{definition}
 
The definition can be extended to random initial data in the usual way. Our next result is the stochastic counterpart of Lemma \ref{lem:bla}. 
\begin{thm}\label{lem:lem}
	Let  $\mathbf{\Omega}=(\Omega,\mathcal{F},(\mathcal{F})_{t},\mathbb{P})$ be a filtered probability space. If $Y$ is a solution to (\ref{eq:Ailang}) on $\mathbf{\Omega}$, then there exists an $\mathbb{R}^{n}$-valued standard Brownian motion $B$ on $\mathbf{\Omega}$, such that $Y$ solves 
	\begin{equation}\label{eq:Ailangeu}
	dY_{t}=g^{-1}(Y_{t})\clr(F(Y_{t}))dt+g^{-1}(Y_{t})\circ dB_{t}.
	\end{equation}
	Vice versa, if $Y$ solves (\ref{eq:Ailangeu}), then there exists an Aitchison diffusion $X$ on $\mathbf{\Omega}$ such that $Y$ solves (\ref{eq:Ailang}).
\end{thm}
 
\begin{proof}
	Let $Y$ be a solution to (\ref{eq:Ailang}). Since $X$ is a Brownian motion on $(\Delta,\langle\cdot,\cdot\rangle_{A})$, we have
	\begin{equation}\label{eq:Aidiffcov}
	\mathbb{E}[\langle X_{t},p\rangle_{A}\langle X_{s},q\rangle_{A}]=\langle p,q\rangle_{A}s\wedge t,
	\end{equation}
	whence, for all $g,h\in H$
	\begin{equation*}
	\mathbb{E}[ (\clr(X_{t}),g)_{H}(\clr( X_{s}),h)_{H}]=(g,h)_{H} s\wedge t.
	\end{equation*}
	Therefore $\tilde{B}:=\clr(X)$ is a BM on $H$. Now take another one-dimensional Brownian motion $W$, independent from $\tilde{B}$ and define $B:=\tilde{B}+\frac{1}{\sqrt{n}}W\mathbf{1}$. Then, because
	\begin{align*}
	&\mathbb{E}[ (B_{t},x)(B_{s},y)]\\
	=&\mathbb{E}\left[ (\tilde{B}_{t},x)(\tilde{B}_{s},y)+n^{-\frac{1}{2}}(\tilde{B}_{t},x)(\mathbf{1},y)W_{s}+n^{-\frac{1}{2}}(\tilde{B}_{s},y)(\mathbf{1},x)W_{t}+n^{-1}W_{s}W_{t}(\mathbf{1},x)(\mathbf{1},y)\right]\\
	=&(\operatorname{pr}_{H}x,\operatorname{pr}_{H}y)_{H} s\wedge t+\frac{1}{n}(\mathbf{1},x)(\mathbf{1},y)s\wedge t=(x,y)s\wedge t,
	\end{align*}
	we know $B$ is a standard BM on $\mathbb{R}^{n}$. Next, set $\tilde{Y}:=\clr(Y)$ and $h:=\clr(p)$. By (\ref{eq:ASDEdef}),
	\begin{equation}
	(\tilde{Y}_{t},h)_{H}=(\tilde{Y}_{0},h)_{H}-\int_{0}^{t}(\clr(F(Y_{s})),h)_{H}ds+(B_{t}-n^{-\frac{1}{2}}W_{t}\mathbf{1},h)_{H}.
	\end{equation}
	Now applying the Stratonovich chain rule to $Y=\sfm(\tilde{Y})$, we recover (\ref{eq:Ailangeu}).\par 
	If on the other hand $Y$ solves (\ref{eq:Ailangeu}), then 
	\begin{align*} 
	d\ln Y^{i}_{t}&=\frac{1}{Y^{i}_{t}}\left((g^{-1}(Y_{t})\clr(F(Y_{t})))_{i}dt+(g^{-1}(Y_{t})\circ dB_{t})_{i}\right)\\
	&=-\left(\clr_{i}(F(Y_{t}))-\sum_{k=1}^{n}Y^{k}_{t}\clr_{k}(F(Y_{t}))\right)dt+dB^{i}_{t}-\sum_{k=1}^{n}Y^{k}_{t}\circ dB^{k}_{t},
	\end{align*}
	and therefore
	\begin{equation*}
	d\tilde{Y}^{i}_{t}:=d\clr_{i}(Y_{t})=-\clr_{i}(F(Y_{t}))dt+dB^{i}_{t}-\frac{1}{n}\sum_{k=1}^{n}dB^{k}_{t}.
	\end{equation*}
	The noise term on the right hand side of the previous equation is a Brownian motion on $H$. Thus, testing $\tilde{Y}$ with $h\in H$ and transforming back to the Aitchison simplex, we see $Y$ obeys (\ref{eq:Ailang}) with $X=\sfm(B)$.  
\end{proof}
      
Choosing again $F=\theta$ in (\ref{eq:Ailangeu}), leads to
 \begin{align}\label{eq:ka}
 dY_{t}=g^{-1}(Y_{t})AY_{t}dt+g^{-1}(Y_{t})\circ dB_{t},
 \end{align}
which is the stochastic replicator equation of Fudenberg and Harris as in (\ref{eq:fudsde}) with $\sigma_{i}=1$ for $i=1,\dots,n$. We will here consider this constant coefficient case only. Yet, our methods could be easily extended to different $\sigma_{i}$, too. In this case, one would need to incorporate a covariance structure in the definition of $X$, by imposing e.g.

\begin{equation}\label{eq:Aidiffcov2}
\mathbb{E}[\langle X_{t},p\rangle_{A}\langle X_{s},q\rangle_{A}]=(\clr( p),\operatorname{diag}(\sigma_{1},\dots,\sigma_{n})\clr(q))s\wedge t.
\end{equation}

There is another interesting and natural choice for $F$ in (\ref{eq:Ailang}). Namely, take a potential ${V\in C^{1}(\Delta)}$ and consider the gradient drift $F=\ominus\nabla^{A}V$. The associated evolution can be interpreted as a Langevin dynamic on the Aitchison simplex:
\begin{equation}\label{eq:Ailang88}
dY_{t}=\ominus\nabla^{A}V(Y_{t})dt\oplus dX_{t},
\end{equation}
which in the standard Euclidean picture corresponds to
\begin{equation}\label{eq:Ailangeu888}
dY_{t}=-\nabla^{G}V(Y_{t})dt+g^{-1}(Y_{t})\circ dB_{t}.
\end{equation}
where $\nabla^{G}:=g^{-1}\nabla^{g}$. We will come back to those gradient drift diffusions in due course. \par 
In preparation for the following two subsections, we need to introduce a classical notion of evolutionary game theory. As for the deterministic replicator equation (\ref{eq:requad}), a key concept in the analysis of the long term behaviour of (\ref{eq:ka}) are Price and Maynard Smith's \textit{evolutionary stable strategies} (ESS) \citep{smith1973logic}, which specify Nash equilibria that are non-invadable by initially rare alternative strategies.

\begin{definition}[Definition](Evolutionary stable strategy)\textbf{.} Let $A=(a_{ij})\in\mathbb{R}^{n\times n}$ be a payoff matrix. We call $p^{*}\in\bar{\Delta}$ an ESS for $A$ provided
	\begin{enumerate}
		\item \textit{equilibrium condition}
		\begin{equation*}
			p\cdot Ap^{*}\le p^{*}\cdot Ap^{*}, \text{ for all } p\in\bar{\Delta}
		\end{equation*} 
		\item \textit{stability condition}
		\begin{equation}\label{eq:stabcond}
		\text{if }p\ne p^{*}\text{ and } p\cdot Ap^{*}=p^{*}\cdot Ap^{*}\text{ then } p\cdot Ap<p^{*}\cdot Ap.
		\end{equation}
	\end{enumerate}
	The first condition above means that $p^{*}$ is a \textit{Nash equilibrium} (NE).
\end{definition}   
   
It is well-known that if $p^{*}\in\Delta$ is an interior ESS, then it must be unique and moreover $A$ is conditionally negative definite \citep[c.f.][]{hofbauer1998evolutionary}:

\begin{definition}
	We call $A\in\mathbb{R}^{n\times n}$ \textit{conditionally negative semi-definite}, whenever,
	\begin{equation*}
	h\cdot Ah\le0,\text{ for all } h\in H.
	\end{equation*} 
	If the previous inequality is strict for all $h\in H\setminus\lbrace 0\rbrace$, we say $A$ is \textit{conditionally negative definite}. We denote by $\Gamma^{\le}=\Gamma^{\le}_{n}$ and $\Gamma^{<}=\Gamma^{<}_{n}$ the sets of all $n\times n$ conditionally negative semi-definite and conditionally negative definite matrices, respectively. 
\end{definition}  
If on the other hand, $A$ is conditionally negative definite, then $A$ has a unique ESS, possibly lying on the boundary.\par 
Finally notice \citep[c.f.][]{imhof2005long}, if $A\in\Gamma^{<}$ and we introduce the Rayleigh-quotient
\begin{equation*}
\lambda:=-\max_{h\in H\setminus\lbrace 0\rbrace}\frac{h\cdot Ah}{\|h\|^{2}},
\end{equation*}
then $\lambda>0$ and for all $h\in H$
	\begin{equation}\label{eq:conneg}
	h\cdot Ah\le-\lambda\|h\|^{2}.
	\end{equation}
The parameter $\lambda$ will play a crucial role in our final section on contraction estimates for replicator dynamics.

\addcontentsline{toc}{subsection}{5.1.\ \ Fokker-Planck equation and invariant measures for stochastic replicator dynamics }   

\bigskip

\textbf{5.1. Fokker-Planck equation and invariant measures for stochastic replicator dynamics}  
\smallskip

\noindent  
In \citep{imhof2005long} Imhof studies the long-run behavior of stochastic replicator dynamics, providing in particular sufficient conditions for the existence of invariant probability measures. The authors of \citep{hofbauer2009time} investigate among others ergodicity properties of (\ref{eq:ka}) and their consequences. Our aim for this subsection is to complement those results by a classical perspective on invariant measures, namely via the Fokker-Planck equation.
\par 
 
Mainly for notational convenience, and in this subsection only, $X$ will generically be an Aitchsion diffusion with covariance structure 
\begin{equation}\label{eq:Aidiffcov}
\mathbb{E}[\langle X_{t},p\rangle_{A}\langle X_{s},q\rangle_{A}]=\sqrt{2}\langle p,q\rangle_{A}s\wedge t,
\end{equation}
which in the Fudenberg-Harris model corresponds to $\sigma_{i}=\sqrt{2},\ i=1,\dots n$, or equivalently to (\ref{eq:ka}) if the Brownian motion obeys $\langle B^{i},B^{j}\rangle_{t}=\sqrt{2}\delta_{ij}t$. We denote by $(P_{t})$ the Markov semi-group associated to (\ref{eq:ka}), which we also refer to as the replicator semigroup. For the corresponding family of Markov transition kernels we write $(\pi_{t})$. Of course, $P_{t}$ depends upon the choice of a payoff matrix $A\in\mathbb{R}^{n\times n}$ and so does the corresponding generator $L_{A}:=L'+Z_{A}$, where $L'=2L$ with $L$ as in (\ref{eq:L}) and for $f\colon\Delta\to\mathbb{R}$ sufficiently regular, 
   \begin{equation*}
   Z_{A}f(p):=Ap\cdot\nabla^{g}f(p).
   \end{equation*}
As usual, we call a $\sigma$-finite (but possibly non-finite) measure $\mu$ on $\Delta$ invariant for the replicator semigroup, if for all positive, bounded and measurable functions $f\colon\Delta\to\mathbb{R}$
\begin{equation}
\int_{\Delta}P_{t}fd\mu=\int_{\Delta}fd\mu, \quad t\ge0.
\end{equation}

Let us start with the following simple, yet important observation which is a direct consequence of the isometry between $\Delta$ and $\mathbb{R}^{n-1}$ and Theorem \ref{lem:lem}.

\begin{lem}[Change of coordinates formula]\label{lem:coc}
	Denote by 
	\begin{equation}\label{eq:thetahat}
	\hat{\theta}(x):=\ilr\circ\theta\circ\ilr^{-1}(x)=\Psi A\sfm(\Psi^{\top}x)
	\end{equation}
	and consider the $\mathbb{R}^{n-1}$-valued diffusion $\hat{Y}$ given by 
	\begin{equation}\label{eq:hatdif}
	d\hat{Y}_{t}=\hat{\theta}(\hat{Y}_{t})dt+\sqrt{2}dB_{t},
	\end{equation} 	
	with corresponding generator $\hat{L}_{A}:=\Delta+\hat{\theta}\cdot\nabla$. The Markov semigroup $(\hat{P}_{t})$ given through (\ref{eq:hatdif}) and the replicator semigroup are linked by means of the change of coordinates
	\begin{equation}\label{eq:semrel}
	P_{t}f(p)=\hat{P}_{t}(f\circ\ilr^{-1})(\ilr(p)).
	\end{equation}  
\end{lem}

Hence, the analysis of $(P_{t})$ eventually boils down to the analysis of the drifted Brownian motion $\hat{Y}$ and its semigroup $(\hat{P}_{t})$. Now recall that the Markov semigroup $(P_{t})$ is regular provided for all $t>0$, the probability measures $\pi_{t}(p,\cdot)$ are mutually equivalent for all $p\in\Delta$ and $t>0$. Then, as first consequence of Lemma \ref{lem:coc}, we obtain

\begin{prop}\label{prop:p}
	The replicator semigroup is regular.
\end{prop}

\begin{proof}
	Notice that, $\hat{\theta}$ is bounded and Lipschitz. Indeed, we have
	\begin{equation*}
	\|\hat{\theta}(x)-\hat{\theta}(y)\|\le\|A\|\|\Psi\|^{2}\|x-y\|,
	\end{equation*}
	and
	\begin{equation}
	\|\hat{\theta}(x)\|\le \|\Psi\|\| A\|
	\end{equation}
	for all $x,y\in\mathbb{R}^{n-1}$. Therefore, $(\hat{P}_{t})$ is strongly Feller and irreducible \citep[Prop. 7.20]{da2006introduction}, hence regular. By Lemma \ref{lem:coc} the replicator semigroup $(P_{t})$ inherits the regularity from $(\hat{P}_{t})$.
\end{proof}

By \citep[Thm 4.2.1]{daprato_zabczyk_1996}, regularity of a semigroup on the other hand directly entails 
\begin{cor}[Uniqueness of invariant probability measures and mixing property]
	 Let ${\mu\in\mathcal{P}(\Delta)}$ be invariant for $(P_{t})$. Then $\mu$ is the only invariant measure and moreover the replicator semigroup is strongly mixing for $\mu$, i.e. for every $p\in\Delta$ and measurable $C\subset\Delta$
	\begin{equation}
	\lim_{t\to\infty}\pi_{t}(p,C)=\mu(C).
	\end{equation}
\end{cor}

Our next aim is a description of invariant measures by means of a stationary Fokker-Planck equation. As a preparatory result we deliver following statement on densities of invariant measures. 

\begin{prop}[Existence of smooth $\lambda_{A}$-densities]
	If $\mu$ is an invariant measure for the replicator semigroup, then $\mu$ has a $C^{\infty}$-smooth density with respect to the Aitchsion measure $\lambda_{A}$.
\end{prop}

\begin{proof}
	By (\ref{eq:semrel}) invariant measures $\mu$ for $(P_{t})$ and $\hat{\mu}$ for $(\hat{P}_{t})$, respectively, are in one to one correspondence via
	\begin{equation}\label{eq:cor}
	\mu=\ilr^{-1}_{\#}\hat{\mu} \qquad\text{ and }\qquad \hat{\mu}=\ilr_{\#}\mu.
	\end{equation}
	Therefore, and because $\lambda_{A}=\ilr^{-1}_{\#}\lambda_{\mathbb{R}^{n-1}}$ it is enough to prove, that every invariant measure $\hat{\mu}$ of $(\hat{P}_{t})$ admits a smooth Lebesgue density. But in fact, suppose $\hat{\mu}$ is an invariant measure. Then it is a positive, weak solution to the stationary Fokker-Planck equation
	\begin{equation}\label{eq:fp}
	0=\hat{L}^{*}_{A}\hat{\mu}:=\Delta\hat{\mu}-\nabla\cdot\left(\hat{\mu}\hat{\theta}\right),
	\end{equation}
	$\hat{L}^{*}_{A}$ being the formal $L^{2}(\lambda_{\mathbb{R}^{n-1}})$-adjoint of $\hat{L}_{A}$. However, since $\hat{L}^{*}_{A}$ has smooth coefficients, by Weyl's regularity theorem \citep[c.f.][Thm 1.4.6]{bogachev2015fokker} any weak solution to (\ref{eq:fp}) is a smooth classical solution. In particular, $\hat{\mu}$ has a smooth density with respect to the Lebesgue measure on $\mathbb{R}^{n-1}$.
\end{proof}

We are now ready to present the main result of this subsection. Let us define for every $p\in\Delta$ the potential
\begin{equation}\label{eq:pot}
\Lambda(p):=\sum_{i=1}^{n}a_{ii}p_{i}-p\cdot Ap.
\end{equation} 
Occassionally, we will also write $\Lambda_{A}$ if we want to emphasize the dependence on the matrix $A$. As we shall see, this potential plays a decisive role for the long time behavior of stochastic replicator dynamics.

\begin{thm}[Stationary Fokker-Planck equation]\label{thm:inv}
	Let $\mu$ be invariant for $(P_{t})$. Then the $\lambda_{A}$ density of $\mu$ is a solution to the stationary Fokker-Planck equation
	\begin{equation}\label{eq:FP}
	0=L'\mu-Z_{A}\mu-\Lambda\mu.
	\end{equation}
	Vice versa, assume $\mu\in C^{2}(\Delta)$ is a strictly positive solution to (\ref{eq:FP}) whose logarithmic gradient is locally Lipschitz and has linear growth, i.e. 
	\begin{equation}\label{eq:growthcond}
	\left\|\nabla^{g}\ln \mu(p)\right\|\le K(1+\|p\|_{A})
	\end{equation}
	for some $K>0$ and all $p\in\Delta$. Then $\mu d\lambda_{A}$ is invariant for $(P_{t})$.
\end{thm}

\begin{proof}
We have to determine $L^{*}_{A}$ the formal $L^{2}(\lambda_{A})$-adjoint of $L_{A}$. Since, $\lambda_{A}$ is reversible for $L'$, we only need to focus the vector field $Z_{A}$ and claim that its formal $L^{2}(\lambda_{A})$-adjoint $Z^{*}_{A}$ is given by
\begin{equation*}
Z^{*}_{A}=-Z_{A}-\Lambda.
\end{equation*}	
Indeed, consider $f,g\in C^{\infty}_{0}(\Delta)$, set $\hat{f}:=f\circ\ilr^{-1}$ and likewise for $\hat{g}$. Then, since 
	\begin{align*}
Z_{A}f=(\Psi A \ilr^{-1},\nabla f\circ\ilr^{-1})\circ\ilr,
\end{align*} 
it follows
\begin{equation}\label{eq:ibp}
\int_{\Delta}gZ_{A}fd\lambda_{A}=\int_{\mathbb{R}^{n-1}}\hat{g}(x)(\Psi A \ilr^{-1}(x),\nabla \hat{f}(x))dx.
\end{equation}
But because
\begin{equation*}
\nabla\cdot\left(\Psi A\ilr^{-1}(x)\right)=\sum_{i=1}^{n}a_{ii}\ilr^{-1}_{i}(x)-\ilr^{-1}(x)\cdot A\ilr^{-1}(x)=\Lambda\circ\ilr^{-1}(x)
\end{equation*}
and $\hat{f},\hat{g}\in C^{\infty}_{0}(\mathbb{R}^{n-1})$, integrating by parts in (\ref{eq:ibp}) yields
\begin{equation}
\int_{\Delta}gZ_{A}fd\lambda_{A}=\int_{\Delta}(-Z_{A}g-\Lambda g)fd\lambda_{A}.
\end{equation}	
Now if $\mu d\lambda_{A}$ is invariant for $(P_{t})$, then for all smooth and compactely supported test functions $f$ on $\Delta$ we have
\begin{equation*}
0=\int_{\Delta}L_{A}f\mu d\lambda_{A}=\int_{\Delta}fL^{*}_{A}\mu d\lambda_{A},
\end{equation*}  
which yields the claim by the Lemma of du Bois-Reymond.  \par 
For the converse direction assume $\mu\in C^{2}(\Delta)$ is a positive solution to (\ref{eq:FP}). Then ${\hat{\mu}:=\mu\circ\ilr^{-1}}$ is a positive solution to
\begin{equation}
0=\hat{L}^{*}_{A}\hat{\mu},
\end{equation}
or in other words, $\hat{\mu}$ is infinitesimally invariant for $(\hat{P}_{t})$. Next, consider the Doob $h$-transform $\hat{\mathcal{L}}$ of $\hat{L}_{A}$ defined by 
\begin{equation}
\hat{\mathcal{L}}f:=\frac{1}{\hat{\mu}}\hat{L}_{A}(\hat{\mu}f)
\end{equation}
and set
\begin{equation}
\hat{\mathcal{L}}^{\dagger}:=\hat{\mathcal{L}}^{*}-\hat{\mathcal{L}}^{*}1. 
\end{equation}
Thus, $\hat{\mathcal{L}}^{\dagger}$ is the formal $L^{2}(\lambda_{\mathbb{R}^{n-1}})$-adjoint of $\hat{\mathcal{L}}$ minus its zero order part and
\begin{equation}
\hat{\mathcal{L}}^{\dagger}f=\Delta f-\hat{\theta}\cdot\nabla f-\nabla\ln\hat{\mu}\cdot \nabla f
\end{equation}
for $f$ sufficiently smooth. Now according to \citep[Thm 4.8.5]{pinsky_1995}, the density $\hat{\mu}$ is invariant for $(\hat{P}_{t})$ iff the diffusion given by the martingale problem for $\hat{\mathcal{L}}^{\dagger}$ is non-explosive. But using the growth condition (\ref{eq:growthcond}) and identifying $x=\ilr(p)$, we know
\begin{equation}
\|\nabla\ln\hat{\mu}(x)\|=\|\nabla^{g}\ln\mu(p)\|\le K(1+\|p\|_{A})=K(1+\|x\|)
\end{equation}
for all $x\in\mathbb{R}^{n-1}$. Likewise, it follows that the logarithmic gradient of $\hat{\mu}$ is locally Lipschitz. Therefore, all coefficients of $\hat{\mathcal{L}}^{\dagger}$ are locally Lipschitz and satisfy a linear growth condition, whence classical theory \citep[c.f.][ch. 6]{ikeda2014stochastic} guarantees apart from well-posedeness of the martingale problem for $\hat{\mathcal{L}}^{\dagger}$, that the associated diffusion is conservative. 
\end{proof}

The rest of this subsection is devoted to the exploitation of the previous theorem. Note that some of the subsequent results were proven already in \citep{hofbauer2009time}. However, whereas Hofbauer and Imhof invoke elaborate Lyapunov function techniques, in our present setting they appear as direct consequences of Theorem \ref{thm:inv}.

\begin{cor}\label{cor:Aiinv}
	The Aitchison measure $\lambda_{A}$ is invariant for $(P_{t})$ if and only if
	the payoff matrix $\mathbb{R}^{n\times n}\ni A=(a_{ij})$ satisfies for all $i\ne j$
	\begin{equation}\label{eq:macon}
	a_{ij}+a_{ji}-a_{ii}-a_{jj}=0.
	\end{equation}
	Moreover, if $n>3$ and $A$ satisfies (\ref{eq:macon}), then the replicator diffusion is transient in that $Y_{t}\to\partial\Delta$ as $t\to\infty$ almost surely.
	
\end{cor}   
Notice, the previous Corollary applies in particular to zero-sum games, that is, when $A=-A^{\top}$.
\begin{proof}   
	Since the Lipschitz and growth condition are trivially satisfied, by Theorem \ref{thm:inv}, $\lambda_{A}$ is invariant iff for all $p\in\Delta$
	\begin{equation}\label{eq:nullc}
	0=\Lambda(p)=\sum_{i=1}^{n}a_{ii}p_{i}-p\cdot Ap.
	\end{equation}
	Observe, if $A$ satisfies (\ref{eq:macon}), then 
	\begin{equation}\label{eq:pap}
	\begin{aligned}
	p\cdot Ap&=\sum_{i\neq j}^{n}p_{i}a_{ij}p_{j}+\sum_{i=1}^{n}a_{ii}p_{i}^{2}=-\sum_{i\neq j}^{n}p_{i}a_{ij}p_{j}+2\sum_{i\neq j}^{n}p_{i}a_{ii}p_{j}+\sum_{i=1}^{n}a_{ii}p_{i}^{2}\\
	&=-p\cdot Ap+2\sum_{i=1}^{n}a_{ii}p_{i},
	\end{aligned}
	\end{equation} 
	whence (\ref{eq:nullc}) is fulfilled. On the other hand, if (\ref{eq:nullc}) is true for all $p\in\Delta$, then by continuity it is valid also for $p\in\bar{\Delta}$. Testing, (\ref{eq:nullc}) with $p=\frac{1}{2}\epsilon_{i}+\frac{1}{2}\epsilon_{j}$ for $i\ne j$ immediately yields (\ref{eq:macon}).\par 
	Regarding the second statement, notice $(\ref{eq:macon})$ entails
	\begin{equation*}
	\nabla\cdot\hat{\theta}(x)=\Lambda\circ\ilr^{-1}(x)=0
	\end{equation*}
	for all $x\in\mathbb{R}^{n-1}$. Therefore, if $n\ge4$ \citep[cor. 6.3]{pinsky_1995} implies $\hat{Y}$ is transient on $\mathbb{R}^{n-1}$ which yields the claim.
\end{proof}

 Occasionally, one is interested in invariant distributions of Gibbs-type
 \begin{equation}\label{eq:inm}
 \mu=e^{-V}\lambda_{A},
 \end{equation}
 for some $V\colon\Delta\to\mathbb{R}$. We denote by $\Gamma$ the carr\'e du champs operator associated to $L'$, that is
 \begin{equation*} 
 \Gamma(f,g):=\langle\nabla^{A}f,\nabla^{A}g\rangle_{A}=(\nabla^{g}f,\nabla^{g}g)=\sum_{i=1}^{n}Z_{i}fZ_{i}g,
 \end{equation*}
 and $\Gamma f:=\Gamma(f,f)$. Then a straight forward application of the diffusion property of $L_{A}$ combined with Proposition \ref{thm:inv} implies
 
 \begin{lem}\label{thm:hj}
 	Let $V\in C^{2}(\Delta)$ and $\nabla^{g}V$ be locally Lipschitz and satisfy the growth condition (\ref{eq:growthcond}). Then, $\mu=e^{-V}\lambda_{A}$ is an invariant measure for $(P_{t})$ if and only if $V$ satisfies
 	\begin{equation}\label{eq:hj}
 	0=L'V-\Gamma V-Z_{A}V+\Lambda.
 	\end{equation}
 \end{lem}
 
 \begin{remark}
 	By classical theory \citep[c.f.][]{bakry2013analysis}, we know that the Langevin dynamic on the Aitchsion simplex
 	\begin{equation}\label{eq:Ailang88}
 	dY_{t}=\ominus\nabla^{A}V(Y_{t})dt\oplus dX_{t},
 	\end{equation}
 	which is associated to the generator $L_{V}:=L'-\Gamma(V,\cdot)$ has an invariant measure of the form (\ref{eq:inm}). Since $L_{A}=L_{V}+Z_{A}+\Gamma(V,\cdot)$, the condition in (\ref{eq:hj}) holds if and only if for all $f\in C^{\infty}_{0}(\Delta)$ we have
 	\begin{equation}\label{eq:intcond}
 	\int_{\Delta}(Z_{A}f+\Gamma(V,f))d\mu=0
 	\end{equation}
 	
 \end{remark}

   \begin{cor}\label{cor: diri}
   	Let $\alpha\in\mathbb{R}^{n}_{>0}$ and set $|\alpha|:=\alpha_{1}+\dots+\alpha_{n}$. The Dirichlet distribution with parameter $\alpha$ is the (unique) invariant measure for $(P_{t})$ if and only if 
   	   
   	\begin{enumerate}[label=(\roman*)]
   		\item the payoff matrix $A=(a_{ij})$ fulfills for $i\ne j$
   		\begin{equation*}
   		a_{ij}+a_{ji}-a_{ii}-a_{jj}=2|\alpha|
   		\end{equation*}
   		and
   		\item $\frac{\alpha}{|\alpha|}$ is a Nash equilibrium for $A$.
   	\end{enumerate}
   	 \end{cor}  
   
\begin{proof}
	  
	The Dirichlet distribution amounts to the choice of
	\begin{equation*}
	V(p):=-\sum_{i=1}^{n}\alpha_{i}\ln p_{i}
	\end{equation*} 
	in (\ref{eq:inm}). Evidently,
	\begin{equation}
	\nabla^{g}V(p)=|\alpha|p-\alpha
	\end{equation}
	satisfies the Lipschitz and growth conditions of Theorem \ref{thm:inv}. Plugging in the definition of $V$ in (\ref{eq:hj}) yields,	
	\begin{equation}\label{eq:dircond}
	0=(\alpha-|\alpha|p)\cdot Ap+\Lambda(p)-\|\alpha-|\alpha|p\|^{2}+|\alpha|\left(1-\|p\|^{2}\right).
	\end{equation}    
	Hence in order to prove Corollary \ref{cor: diri}, by Lemma \ref{thm:hj} it is enough to show that (\ref{eq:dircond}) holds for all $p\in\Delta$ if and only if the conditions $(i)$ and $(ii)$ of Corollary \ref{cor: diri} are fulfilled.\par 
	Let us start by assuming that $A$ and $\alpha$ are such that $(i)$ and $(ii)$ are valid. It was shown in \citep{hofbauer2009time} that a matrix for which $(i)$ holds obeys 
	\begin{equation}\label{eq:hcondi}
	h\cdot Ah=-|\alpha|\|h\|^{2}
	\end{equation}
	for every $h\in H$.
	Therefore, and because $\alpha/|\alpha|$ is an interior NE for $A$ by $(ii)$, we know that for any $p\in\Delta$
	\begin{align*}
	&(\alpha-|\alpha|p)\cdot Ap-\|\alpha-|\alpha|p\|^{2}\\
	=&-|\alpha|\left(\frac{\alpha}{|\alpha|}-p\right)\cdot A\left(\frac{\alpha}{|\alpha|}-p\right)-\|\alpha-|\alpha|p\|^{2}+|\alpha|\left(\frac{\alpha}{|\alpha|}-p\right)\cdot A\left(\frac{\alpha}{|\alpha|}\right)=0
	\end{align*}
	Next, observe that arguing akin to (\ref{eq:pap}), we see that for every $p\in\Delta$      
	\begin{equation}\label{eq:w2}
	p\cdot Ap=|\alpha|\left(1-\|p\|^{2}\right)+\sum_{i=1}^{n}a_{ii}p_{i},
	\end{equation}
	which yields (\ref{eq:dircond}) for all $p\in\Delta$. Note that the previous identity can be rephrased as 
	\begin{equation}\label{eq:w3}
	\Lambda_{A}=\Lambda_{-|\alpha|\id}.
	\end{equation}
	Now assume on the contrary that (\ref{eq:dircond}) holds for all $p\in\Delta$. Testing the equation with $\epsilon_{i}$, the $i$-th unit vector in $\mathbb{R}^{n}$ (more precisely, take a sequence $(p_{n})_{n\ge0}\subset\Delta$ with $p_{n}\to\epsilon_{i}$ as $n\to\infty$, test (\ref{eq:dircond}) with $p_{n}$ and take limits) yields
	
	\begin{equation}\label{eq:w1}
	0=\sum_{k=1}^{n}a_{ki}\alpha_{k}-|\alpha|a_{ii}-\|\alpha\|^{2}+2|\alpha|\alpha_{i}-|\alpha|^{2}.
	\end{equation}
	Next, testing with $\frac{1}{2}\epsilon_{i}+\frac{1}{2}\epsilon_{j}$ we find
	\begin{align*}
	0&=\sum_{k=1}^{n}(a_{ki}+a_{kj})\alpha_{k}-\frac{|\alpha|}{2}(a_{ii}+a_{ij}+a_{ji}+a_{jj})+a_{ii}+a_{jj}-\frac{1}{2}(a_{ii}+a_{ij}+a_{ji}+a_{jj})\\
	&-2\|\alpha\|^{2}+2|\alpha|(\alpha_{i}+\alpha_{j})-|\alpha|^{2}+|\alpha|.
	\end{align*}
	Using (\ref{eq:w1}), the previous expression simplifies to
	\begin{equation*}
	0=|\alpha|^{2}+|\alpha|-\frac{|\alpha|+1}{2}(a_{ij}+a_{ji}-a_{ii}-a_{jj}),
	\end{equation*}
	which thus gives condition $(i)$. But then, due to (\ref{eq:hcondi}), (\ref{eq:w2}) and by assumption
	
	\begin{equation*}
	|\alpha|\left(p-\frac{\alpha}{|\alpha|}\right)\cdot A\frac{\alpha}{|\alpha|}=\Lambda(p)+|\alpha|\left(1-\|p\|^{2}\right)=0
	\end{equation*}
	for every $p\in\Delta$. Hence $\frac{\alpha}{|\alpha|}$ is a NE for $A$.
\end{proof}
        
\begin{remark}  
Using (\ref{eq:conneg}) it is not hard to see that for every conditionally negative definite payoff matrix $A\in\mathbb{R}^{n\times n}$, there (uniquely) exists a Euclidean distance matrix $D$ \citep[c.f.][]{krislock2012euclidean} such that
\begin{equation}\label{eq:gencase}
\Lambda_{A}(p)=-\lambda(1-\|p\|^{2})-\frac{1}{2}p\cdot Dp=\Lambda_{-\lambda\id+\frac{1}{2}D}(p).
\end{equation} 
By \citep[Thm 3.1]{hofbauer2009time} the mean of an invariant measure for the replicator semigroup constitutes a Nash equilibrium, say $p^{*}$, for $A$. But then in the light of Corollary \ref{cor: diri} and in particular (\ref{eq:w3}), the term $-\lambda(1-\|p\|^{2})$ in (\ref{eq:gencase}) corresponds to a Dirichlet distribution with parameter $\lambda p^{*}$. This suggests that for general $A\in\Gamma^{<}$ invariant measures of stochastic replicator dynamics will be perturbations or generalizations of the Dirichlet family in which the matrix $D$ enters as an additional parameter. Whether there exists an explicit expression for those measures is left as an interesting question for further investigations.
\end{remark}

 \addcontentsline{toc}{subsection}{5.2.\ \ Wasserstein contractions for stochastic replicator dynamics }

\textbf{5.2. Wasserstein contractions for stochastic replicator dynamics}\par\smallskip 
 
\noindent 

Corollary \ref{cor: diri} in the previous subsection provided us with necessary and sufficient conditions for a stochastic replicator dynamic to attain the Dirichlet distribution $Dir_{\alpha}$ as invariant measure.\par 
Another diffusion process on the Aitchison simplex for which $Dir_{\alpha}$ is invariant (in fact reversible) is the Langevin equation

	\begin{equation}\label{eq:la}
dY_{t}=\ominus\nabla^{A}V(Y_{t})dt\oplus dX_{t},
\end{equation}
with
\begin{equation*}
V(p)=-\sum_{i=1}^{n}\alpha_{i}\ln p_{i}.
\end{equation*}
Recall that $W_{A}$ is the natural Wasserstein distance on $(\Delta,\langle\cdot,\cdot\rangle_{A})$ with cost $d^{2}_{A}$ (see (\ref{eq:Wadi})). Now let us also introduce 
\begin{equation}\label{eq:watilde}
W^{2}_{\Delta}(\mu,\nu):=\inf_{\pi\in\Pi(\mu,\nu)}\int_{\Delta\times\Delta}\|p-q\|^{2}\pi(dpdq).
\end{equation}

 The following proposition is the motivation for our subsequent investigations.
	\begin{prop}\label{pro:dircon}
	Consider the Langevin diffusion $Y=(Y_{t})$ as given through (\ref{eq:la}) and denote by $\mu_{t}:=Law(Y_{t})$. Then,
	\begin{equation}\label{eq:wacon11}
	W_{A}(\mu_{t},Dir_{\alpha})\le W_{A}(\mu_{0},Dir_{\alpha})
	\end{equation}
	and moreover,
	\begin{equation}\label{eq:wacon12}
	W_{\Delta}(\mu_{t},Dir_{\alpha})\le e^{-|\alpha|t}W_{A}(\mu_{0},Dir_{\alpha}).
	\end{equation}
\end{prop}

\begin{proof}
	We first show that $V(p)=-\sum\alpha_{i}\ln p_{i}$ is a convex function on the Aitchison simplex. Since $\nabla^{g}V(p)=|\alpha|p-\alpha$, we have
	\begin{equation*}
	\langle\nabla^{A}V(p)\ominus\nabla^{A}V(q),p\ominus q\rangle_{A}=|\alpha|(p-q,\clr(p)-\clr(q)).
	\end{equation*}
	Now replacing $p=\sfm(x)$ and $q=\clr(y)$ yields
	\begin{equation*}
	\langle\nabla^{A}V(p)\ominus\nabla^{A}V(q),p\ominus q\rangle_{A}=|\alpha|(\sfm(x)-\sfm(y),x-y)\ge 0
	\end{equation*}
	by monotonicity of the softmax function (\citep[c.f.][]{gao2017properties}) and hence $V$ is convex. Now take two solutions $Y$ and $Y'$ to (\ref{eq:Ailang}), both driven by the same noise $X$. Then
	\begin{equation*}
	\frac{d}{dt}\|Y_{t}\ominus Y'_{t}\|^{2}_{A}=-2\langle\nabla^{A}V(Y_{t})\ominus\nabla^{A}V(Y_{t}'),Y_{t}\ominus Y_{t}'\rangle_{A}\le 0
	\end{equation*}
	which immediately implies (\ref{eq:wacon11}).\par 
	We move on to the \enquote*{mixed} Wasserstein contraction claimed in (\ref{eq:wacon12}). First observe that $\sfm$ satisfies a stronger property then just being monotone, namely the softmax function is \textit{co-coercive} \citep[c.f.][]{gao2017properties}, i.e.
	\begin{equation}
	(\sfm(x)-\sfm(y),x-y)\ge\|\sfm(x)-\sfm(y)\|^{2},\quad x,y\in\mathbb{R}^{n}.
	\end{equation}
	Therefore,
	\begin{equation*}
	\frac{d}{dt}\|Y_{t}\ominus Y'_{t}\|^{2}_{A}\le -2|\alpha|\|Y_{t}-Y_{t}'\|^{2}
	\end{equation*}
	and using (\ref{eq:normequi}) we find
	
	\begin{equation*}
	\|Y_{t}-Y_{t}'\|^{2}_{2}\le\|Y_{t}\ominus Y_{t}'\|^{2}_{A}\le\|Y_{0}\ominus Y_{0}'\|^{2}_{A}-2|\alpha|\int_{0}^{t}\|Y_{s}-Y_{s}'\|^{2}_{2}ds.
	\end{equation*}
	Thus, by Gronwall's inequality
	\begin{equation*}
	\|Y_{t}-Y_{t}'\|^{2}_{2}\le e^{-2|\alpha|}\|Y_{0}\ominus Y_{0}'\|^{2}_{A},
	\end{equation*}
	which yields the claim after minimising on both sides over all couplings.
\end{proof}

	Of course, the result just proven is actually stronger then the inequalities stated in (\ref{eq:wacon11}) and (\ref{eq:wacon12}). Indeed, within these two estimates we may replace $Dir_{\alpha}$ by any other (law of a) solution to (\ref{eq:la}), say $(\mu_{t}')$. Then Proposition \ref{pro:dircon} asserts that with respect to $W_{\Delta}$ such laws attract exponentially fast, irrespective of the initial data. \par 
	
	  The question we pose now is: can we find payoff matrices $A$ which enforce a synchronization in the relaxation to equilibrium between replicator diffusion and the Langevin dynamic (\ref{eq:la})? Or in other words, for which payoff matrices can we monitor for stochastic replicator dynamics the same contraction behavior as the one in Proposition \ref{pro:dircon}? 
	
 In order to answer this question, let us first dwell upon the deterministic setting.

	\begin{thm}\label{thm:ecnsd}
		Let $A\in\mathbb{R}^{n\times n}$ be a payoff matrix. The following four statements are equivalent:
		\begin{enumerate}[label=(\roman*)]
			\item  For every $p,q\in\Delta$ let $(p(t))$  and $(q(t))$ be two solutions to the deterministic replicator equation (\ref{eq:requad}) starting in $p$ and $q$, respectively. Then,
			\begin{equation}\label{eq:e1}
			\|p(t)\ominus q(t)\|_{A}\le\|p\ominus q\|_{A},\quad t\ge 0
			\end{equation}
			\item The map $p\mapsto\ominus\theta(p)=\theta(-p)$ is monotone, i.e.
			\begin{equation}\label{eq:e2}
			\langle\theta(p)\ominus\theta(q),p\ominus q\rangle_{A}\le 0,\quad p,q\in\Delta.
			\end{equation}
			\item For every $p\in\Delta$ and $h\in H$
			\begin{equation}\label{eq:e3}
			h\cdot Ag^{-1}(p)h\le 0.
			\end{equation}		
			\item There exist $\lambda\ge0$ and vectors $u,v\in\mathbb{R}^{n}$ such that
			\begin{equation}\label{eq:e5}
			A=-\lambda\id+u\otimes\mathbf{1}+\mathbf{1}\otimes v
			\end{equation}
		\end{enumerate}	
	\end{thm}

	\begin{proof} 
		The chain of implications from $(i)$ to $(iii)$ is fairly standard and we only indicate the key ideas. 
		(i)$\implies$(ii): differentiate $\|p(t)\ominus q(t)\|_{A}^{2}$ at $t=0$. (ii)$\implies$(iii): observe that (\ref{eq:e2}) is equivalent to 
		\begin{equation}
		(A\sfm(x)-A\sfm(y),x-y)\le 0,\quad x,y\in H.
		\end{equation}
		Thus, for every $h\in H$ and $\tau>0$, one has
		\begin{equation*}
		(A\sfm(x+\tau h)-A\sfm(x),h)\le 0.
		\end{equation*}
		Dividing by $\tau$ and taking $\tau\searrow 0$ yields (\ref{eq:e3}).\par 
		 We are now proving (iii)$\implies$(iv), thereby starting with the cases of dimensions $n=2$. First observe that (\ref{eq:e3}) necessitates that $A$ is conditionally negative semi-definite, which can be seen by choosing $p=n^{-1}\mathbf{1}$. Therefore, if $n=2$ and 
		\begin{equation*}
		A=\left( \begin{array}{rr}a & b \\ c & d \\\end{array}\right),
		\end{equation*}
		we know $a+d\le b+c$. Then, the claim follows by taking $2\lambda=b+c-(a+d)$, $2u=(b-d,c-a)$ and $2v=(a+c,b+d)$.\par 
		Let us now we consider the case $n\ge 3$ and write $a^{i}:=A\epsilon_{i}$ for the $i$-th column of $A$. Notice that by continuity (\ref{eq:e3}) holds for all $p\in\bar{\Delta}$. Testing with $2p=\epsilon_{i}+\epsilon_{j}$, for $i\ne j$, we learn that $A$ must satisfy the peculiar monotonicity-like property
		\begin{equation}\label{eq:necco}
		(h_{i}-h_{j})\sum_{k=1}^{n}(a_{ki}-a_{kj})h_{k}\le 0
		\end{equation}
		for all $h\in H$ and $i\ne j$. By continuity we infer
		\begin{equation}\label{eq:moneq}
		h\in H \text{ with }	h_{i}=h_{j} \implies (A^{\top}h)_{i}=(A^{\top}h)_{j}.
		\end{equation}
		Equivalent to the previous implication is the fact that $h\in\langle\lbrace\mathbf{1},\epsilon_{i}-\epsilon_{j}\rbrace\rangle^{\bot}$ entails \linebreak ${h\in\langle a^{i}-a^{i}\rangle^{\bot}}$. Thus, we can find scalars $s_{ij},t_{ij}$ such that
		\begin{equation}
		a^{i}-a^{j}=s_{ij}\mathbf{1}+t_{ij}(\epsilon_{i}-\epsilon_{j}).
		\end{equation}
		Then, if $n=3$ writing $a^{i}-a^{j}$=$a^{i}-a^{k}+a^{k}-a^{j}$ for distinct $i,j,k$ and since $\mathbf{1}=\epsilon_{i}+\epsilon_{j}+\epsilon_{k}$ is follows
		\begin{equation}
		(\tilde{s}+t_{ij}-t_{ik})\epsilon_{i}+(\tilde{s}+t_{kj}-t_{ij})\epsilon_{j}+(\tilde{s}+t_{ik}-t_{kj})\epsilon_{k}=0,
		\end{equation}
		where $\tilde{s}:=(s_{ij}-s_{ik}+s_{kj})$. By linear independence we deduce that $t_{ij}=:t$ must not depend on the indices and $\tilde{s}=0$, which moreover entails $s_{ij}=v_{i}-v_{j}$ for some $v\in\mathbb{R}^{3}$. Likewise, for $n\ge 4$ the linear independence of $\mathbf{1},\epsilon_{i}-\epsilon_{k},\epsilon_{i}-\epsilon_{j},\epsilon_{j}-\epsilon_{k}$ for pairwise distinct $i,j,k$, yields independence of $t_{ij}$ on the indices as well as $s_{ij}=v_{i}-v_{j}$ for some $v\in\mathbb{R}^{n}$.
		At last, consider the vector $u:=a^{i}-t\epsilon_{i}-v_{i}\mathbf{1}$ and observe that it does not depend on $i$. Thus, it follows $A=t\operatorname{id}+u\otimes\mathbf{1}+\mathbf{1}\otimes v$. Using (\ref{eq:e3}) it is easy to see that $-\lambda:=t\le0$, which proves the claim.\par 
		We end the proof by showing (iv)$\implies$(i). First notice, if $A$ satisfies (\ref{eq:e5}) then $(Ap)_{i}=-\lambda p_{i}+u_{i}+p\cdot v$. Hence, if $p(t)$ is a solution to the replicator equation (\ref{eq:requad})
		\begin{equation}
		\frac{d}{dt}\clr_{i}p(t)=-\lambda p_{i}(t)+u_{i}-\frac{1}{n}\sum_{k=1}^{n}u_{k}+\lambda.
		\end{equation}
		But then it follows, that if $(q(t))$ is another solution to (\ref{eq:requad}) 
		\begin{equation*}
		\frac{d}{dt}\|p(t)\ominus q(t)\|^{2}_{A}=-2\lambda(\clr(p(t))-\clr(q(t)),p(t)-q(t))\le 0
		\end{equation*}
		which yields (\ref{eq:e1}).	
	\end{proof}  
	 
	Note that for the case $\lambda=0$, i.e. when $A=u\otimes\mathbf{1}+\mathbf{1}\otimes v$, we only know $A\in\Gamma^{\le}$ from which one cannot infer the existence of ESS. However, due to the simple structure of such payoff matrices, one can easyly give a full characterization of Nash equilibria and ESS. In fact, the following proposition follows straight forward from the definitions of NE and ESS. 
	 
	\begin{prop}\label{prop:NA}
		Denote by $N(A)$ the set of all Nash equilibria of $A$.	If $A=u\otimes\mathbf{1}+\mathbf{1}\otimes v$, then
\begin{equation*}
N(A)=\argmax_{p\in\bar{\Delta}}p\cdot u.
\end{equation*}
Moreover, such $A$ has an ESS iff $N(A)$ is a singleton (i.e. when $u$ has a distinct maximal entry).
	\end{prop}

	   \begin{exa}
	   	For the matrix
	   	\begin{equation}\label{eq:telema}
	   	A=\left( \begin{array}{rrr}1 & 2 & 3\\ 4 & 5 & 6 \\7 & 8 &9 \end{array}\right)
	   	\end{equation}
	   	the pure strategy $p^{*}=(0,0,1)$ is the unique NE and ESS, because
	   	\begin{equation*}
	   	A=(0,3,6)\otimes\mathbf{1}+\mathbf{1}\otimes (1,2,3).
	   	\end{equation*} 
	   \end{exa}
	   
If $\lambda>0$ in (\ref{eq:e5}), then $A\in\Gamma^{<}$ and we know $A$ has an ESS. Clearly, if $u\in\langle\mathbf{1}\rangle$ then the barycenter $e$ is an interior Nash equilibrium and also ESS. Otherwise, a necessary and sufficient condition on $\lambda$ ensuring the existence of interior NE is given in
\begin{prop}\label{prop:hmpf}
	Let $A$ satisfy (\ref{eq:e5}) with $u\notin\langle\mathbf{1}\rangle$. Denote $|u|:=u_{1}+\dots+u_{n}$ and for $x\in\mathbb{R}$ set $x^{-}:=-\min(x,0)$.
	Then $A$ has an interior Nash equilibrium if and only if $\lambda>|u|+n\max_{i}u^{-}_{i}$.
\end{prop}

\begin{proof}
	Consider first the case when $u\in\mathbb{R}^{n}_{\ge 0}$. Assume $\lambda>|u|$. Then 
	\begin{equation*}
	\delta:=\frac{1}{n}\left(1-\frac{|u|}{\lambda}\right)>0.
	\end{equation*}
	Define
	\begin{equation*}
	p_{i}^{*}:=\frac{u_{i}}{\lambda}+\delta.
	\end{equation*}
	Then $p^{*}\in\Delta$ and 
	\begin{equation}\label{eq:NE}
	-\lambda p_{1}^{*}+u_{1}=\dots=-\lambda p_{n}^{*}+u_{n},
	\end{equation}
	whence $p^{*}$ is an interior NE.
	For the converse direction, suppose $p^{*}$ is an interior NE. Then $p^{*}$ obeys (\ref{eq:NE}). Let $i^{*}$ be such that $u_{i^{*}}=\min_{i}u_{i}$. Using (\ref{eq:NE}) it follows
	\begin{equation*}
	n(-\lambda p^{*}_{i^{*}}+u_{i^{*}})=-\lambda+|u|
	\end{equation*}
	and therefore 
	\begin{equation*}
	(1-np^{*}_{i^{*}})\lambda=|u|-nu_{i^{*}}>0.
	\end{equation*}
	Hence, $(1-np^{*}_{i^{*}})\in(0,1)$ and
	\begin{equation*}
	\lambda\ge\frac{|u|}{1-np^{*}_{i^{*}}}>|u|.
	\end{equation*}
	Now we drop the sign condition on $u$ and consider some general $u\in\mathbb{R}^{n}$. Recall, that Nash equilibria for a payoff matrix $A$ are invariant under the addition of a constant to any of the columns of $A$. Therefore,
	\begin{equation}
	N(A)=N(-\lambda\operatorname{id}+u\otimes\mathbf{1})=N(-\lambda\operatorname{id}+\tilde{u}\otimes\mathbf{1}),
	\end{equation}
	where $\tilde{u}\in\mathbb{R}^{n}_{\ge0}$ is obtained from $u$ by 
	\begin{equation}
	\tilde{u}_{i}:=u_{i}+\max_{i} u^{-}_{i}
	\end{equation}
	and we can apply the result of the previous setting.
\end{proof}  
 
Figure 2 depicts phase portraits of replicator dynamics corresponding to $A$ as in (\ref{eq:telema}) for (a), $A-5\operatorname{id}$ in (b), and $A-10\operatorname{id}$ in (c). The corresponding ESS are $(0,0,1)$ in (a) and computed numerically using \citep{dynamo}, $(0,\frac{1}{5},\frac{4}{5})$ in (b) and $(\frac{1}{30},\frac{1}{3},\frac{19}{30})$ in (c).\par
 
\begin{figure}% 
	\centering
	\subfloat[][]{\includegraphics[width=0.3\linewidth]{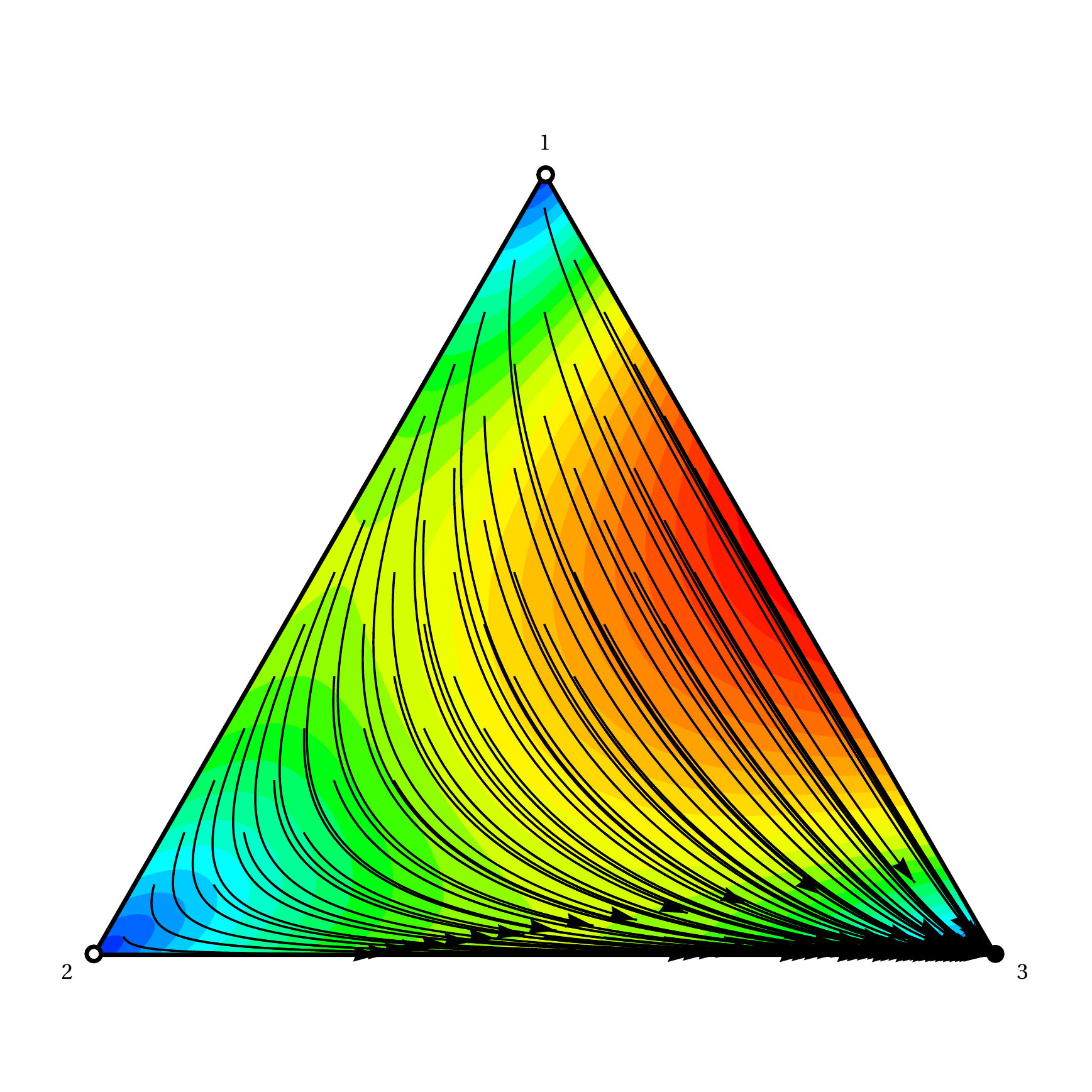}}%
	\quad
	\subfloat[][]{\includegraphics[width=0.3\linewidth]{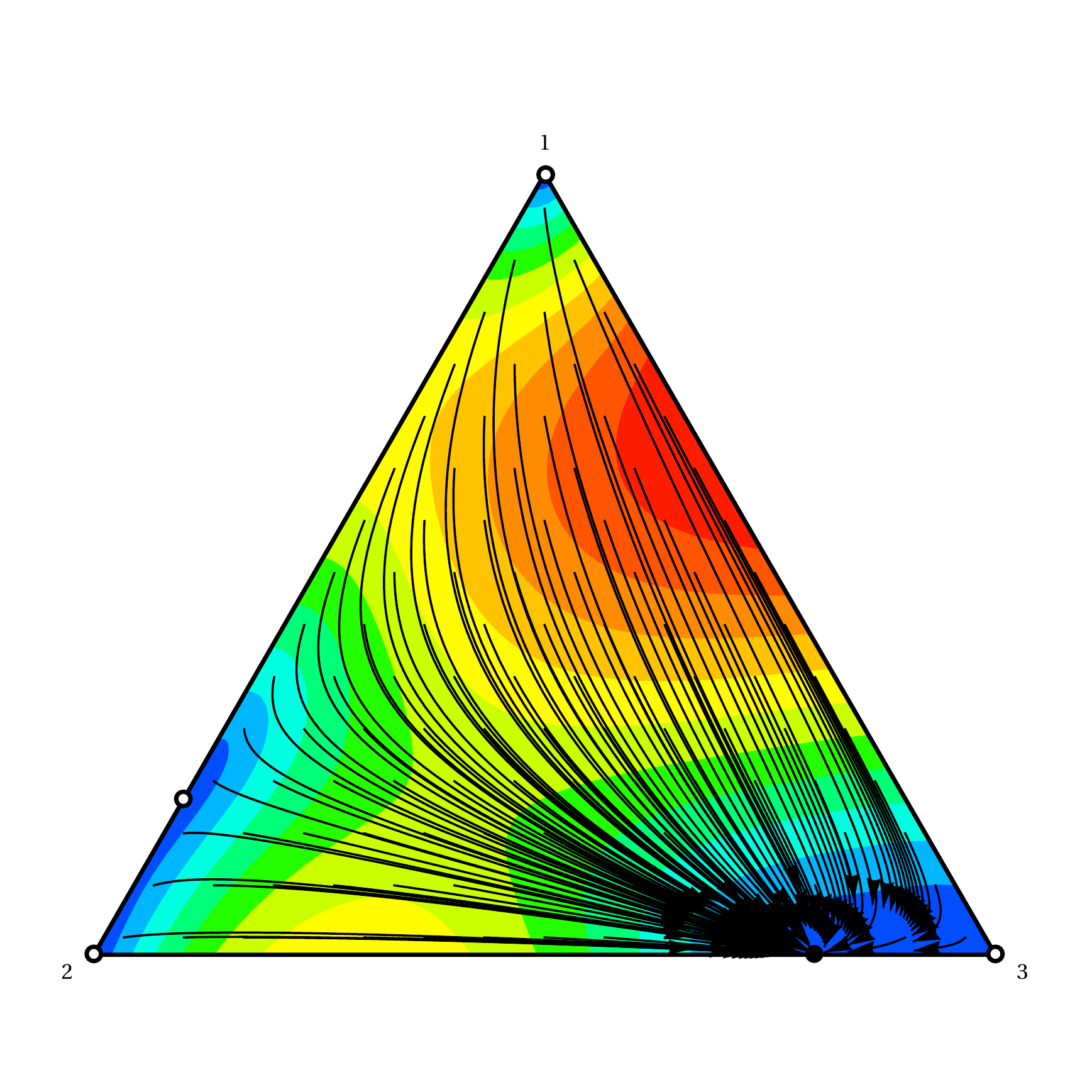}}%
	\quad
	\subfloat[][]{\includegraphics[width=0.3\linewidth]{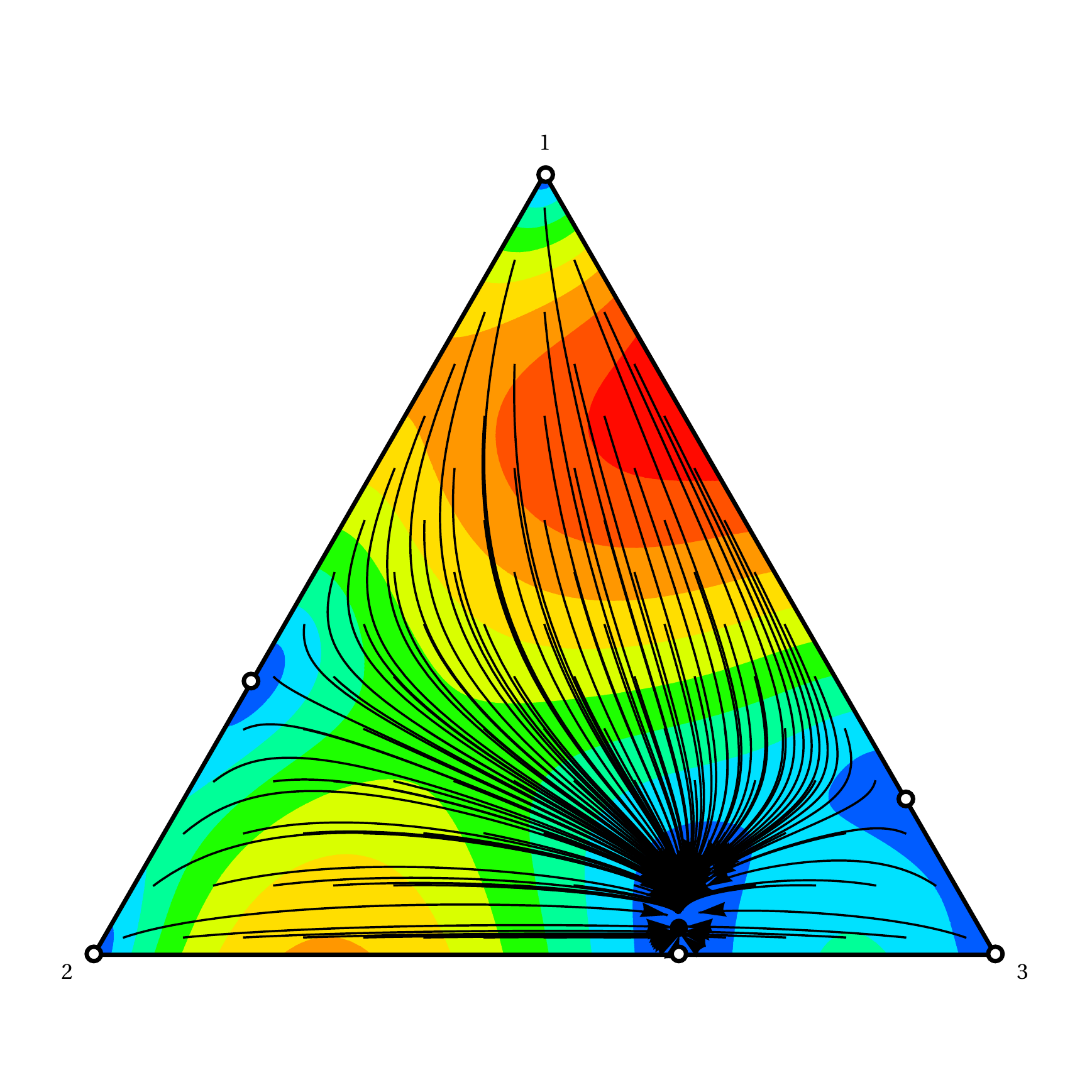}}
	\caption{Phase portraits of replicator dynamics for increasing choice of $\lambda$.}%
\end{figure}

Finally, observe that for such matrices one can update the estimate in (\ref{eq:e1}) to an exponential contraction:	   
	   
	 \begin{cor}\label{cor:ecnsd}
		Let $A\in\mathbb{R}^{n\times n}$ be a payoff matrix with 
		\begin{equation*}
		A=-\lambda\id+u\otimes\mathbf{1}+\mathbf{1}\otimes v
		\end{equation*}
		for some vectors $u,v\in\mathbb{R}^{n}$ and $\lambda>0$. Then, if $(p(t))$ and $(q(t))$ are solutions to the deterministic replicator equation (\ref{eq:requad}), the following estimates are valid
		
		\begin{equation}\label{eq:c2}
		\|p(t)\ominus q(t)\|^{2}_{A}+2\lambda\int_{0}^{t}\|p(s)-q(s)\|^{2}ds\le\|p\ominus q\|^{2}_{A},\quad t\ge 0
		\end{equation}
		and
		\begin{equation*}
		\|p(t)-q(t)\|\le e^{-\lambda t}\|p\ominus q\|.
		\end{equation*}
		
		However, there does not exist any matrix $A\in\mathbb{R}^{n\times n}$ for which one can find some $\lambda>0$ such that
		\begin{equation}\label{eq:con}
		\|p(t)\ominus q(t)\|_{A}\le e^{-\lambda t}\|p\ominus q\|_{A}.
		\end{equation}
	\end{cor}
	
	\begin{proof}
		As in the last part of the proof of Theorem \ref{thm:ecnsd}, we have
		\begin{equation*}
		\frac{d}{dt}\|p(t)\ominus q(t)\|^{2}_{A}=-2\lambda(\clr(p(t))-\clr(q(t)),p(t)-q(t))\le-2\lambda\|p(t)-q(t)\|^{2},
		\end{equation*}
		which immediately entails (\ref{eq:c2}). On the other hand, since $\|p-q\|\le\|p\ominus q\|_{A}$ we also find
		\begin{equation*}
		\|p(t)-q(t)\|^{2}\le\|p\ominus q\|^{2}-2\lambda\int_{0}^{t}\|p(s)-q(s)\|^{2}ds,
		\end{equation*}
		from which we infer (\ref{eq:c2}) by Gronwall's inequality.\par 
		As for the second part of the claim let $A$ be an arbitrary $n\times n$ payoff matrix and, aiming for a contraction, assume there is $\lambda>0$ such that (\ref{eq:con}) holds true. Then, differentiation at $t=0$ yields that $\ominus\theta$ ought to be $\lambda$-strongly monotone, i.e. for all $p,q\in\Delta$
		\begin{equation*}
		\langle\theta(p)\ominus\theta(q),p\ominus q\rangle_{A}\le-\lambda \|p\ominus q\|^{2},
		\end{equation*}
		or equivalently
		\begin{equation*}
		(A\sfm(x)-A\sfm(y),x-y)\le-\lambda\|x-y\|^{2},\quad x,y\in H.
		\end{equation*}
		Now rescale the previous inequality by replacing $x$ and $y$ by $\beta x$ and $\beta y$ for some $\beta\in\mathbb{R}$. Then,
		\begin{equation}\label{eq:x}
		\frac{1}{\beta}(A\sfm(\beta x)-A\sfm(\beta y),x-y)\le-\lambda\|x-y\|^{2}
		\end{equation}
		Now choose, $x\ne y$ both such that they have one distinct maximal entry. Then \citep[c.f.][]{gao2017properties},
		\begin{equation*}
		\lim_{\beta\to\infty}\sfm(\beta x)=\argmax x.
		\end{equation*}
		Therefore taking $\beta$ to infinity in (\ref{eq:x}) yields
		\begin{equation*}
		0\le-\lambda\|x-y\|^{2}
		\end{equation*}
		contradicting our assumption on the sign of $\lambda$.
	\end{proof}

Whereas the previous findings might be of independent interest for evolutionary game theory, the main reason for treating in depth the deterministic dynamic is that these results have an immediate counterpart in the stochastic setting. \par

\begin{thm}[Wasserstein contractions for stochastic replicator dynamics]\label{thm:Wacontract}
Let $A\in\mathbb{R}^{n\times n}$ be a payoff matrix and $Y$ and $Y'$ be a solutions to the stochastic replicator equation (\ref{eq:ka}) with $Y_{0}\sim\mu_{0}$ and $Y_{0}'\sim\mu_{0}'$. Denote for every $t\ge0$ by $\mu_{t}$ and $\mu_{t}'$ the law of $Y_{t}$ and $Y_{t}'$, respectively. Then,
\begin{equation}\label{eq:wacon21}
W_{A}(\mu_{t},\mu_{t}')\le W_{A}(\mu_{0},\mu_{0}'),\quad t\ge 0.
\end{equation}
if and only if $A=-\lambda\id+u\otimes\mathbf{1}+\mathbf{1}\otimes v$ for some vectors $u,v\in\mathbb{R}^{n}$ and $\lambda\ge 0$. If, $\lambda>0$ then additionally
\begin{equation}\label{eq:wacon21'}
W_{\Delta}(\mu_{t},\mu_{t}')\le e^{-\lambda t}W_{A}(\mu_{0},\mu_{0}'),\quad t\ge 0.
\end{equation}

However, there is no matrix $A$ such that for some $\lambda>0$
\begin{equation}\label{eq:wacon22}
W_{A}(\mu_{t},\mu_{t}')\le e^{-\lambda t} W_{A}(\mu_{0},\mu_{0}'),\quad t\ge 0.
\end{equation}

\end{thm} 

\begin{proof}
	Using Theorem \ref{lem:lem}, we can represent a solution to the stochastic replicator equation $Y$ by the SDE on the Aitchison simplex:
	\begin{equation}\label{eq:ai}
	dY_{t}=\theta(Y_{t})dt\oplus dX_{t}.
	\end{equation}
	In particular, if $Y$ and $Y'$ are driven by the same Brownian motion $B$, then there is an Aitchison diffusion $X$, driving both $Y$ and $Y'$ in their Aitchison representation (\ref{eq:ai}). Now if $A$ satisfies (\ref{eq:e5}), then by monotonicity of $\ominus\theta$ (c.f. \ref{eq:e2}) if follows
	\begin{equation}\label{eq:whatever}
	\frac{d}{dt}\|Y_{t}\ominus Y_{t}'\|^{2}_{A}=2\langle\theta(Y_{t})\ominus\theta(Y_{t}'),Y_{t}\ominus Y_{t}'\rangle_{A}\le 0
	\end{equation}
	and hence, $\mathbb{E}\|Y_{t}\ominus Y_{t}'\|^{2}_{A}\le\mathbb{E}\|Y_{0}\ominus Y_{0}'\|^{2}_{A}$, which yields (\ref{eq:wacon21}) after optimizing over all coupling on both sides of the inequality.\par 
	For the other direction, take $(\mu_{t})$ and $(\mu_{t}')$ with $\mu_{0}=\delta_{p}$ and $\mu'_{0}=\delta_{q}$, where $p,q\in\Delta$. By assumption, we have
	\begin{equation}\label{eq:pel}
	W_{\mathbb{R}^{n-1}}(\ilr^{-1}_{\#}\mu_{t},\ilr^{-1}_{\#}\mu_{t}')\le W_{\mathbb{R}^{n-1}}(\ilr^{-1}_{\#}\delta_{p},\ilr^{-1}_{\#}\delta_{q})
	\end{equation}
	Since, the generator $\hat{L}$ of the process $\hat{Y}=\ilr(Y_{t})$ is given by
	\begin{equation*}
	\hat{L}f(x)=\frac{1}{2}\Delta f(x)+\nabla f(x)\cdot\hat{\theta}(x)
	\end{equation*}
	it is well-known \citep[c.f.][]{Natile2011,bolley2012convergence,von2005transport}
that (\ref{eq:pel}) implies (indeed is equivalent to)
	\begin{equation*}\label{eq:jjj}
	(\hat{\theta}(x)-\hat{\theta}(y),x-y)\le0,
	\end{equation*}
	for all $x,y\in\mathbb{R}^{n-1}$. This in turn is equivalent to $\ominus\theta$ being monotone and thus, by Theorem \ref{thm:ecnsd} $A$ obeys (\ref{eq:e5}).\par 
	If we know that $\lambda>0$, we can improve (\ref{eq:whatever}) to
		\begin{equation}\label{eq:whatever}
	\frac{d}{dt}\|Y_{t}\ominus Y_{t}'\|^{2}_{A}\le-2\lambda\|Y_{t}-Y_{t}'\|^{2}
	\end{equation}
	which yields (\ref{eq:wacon21'}) by the same arguments we used in the proof of Corollary \ref{cor:ecnsd}. Finally, again by e.g. \citep[][]{Natile2011} the exponential contraction in (\ref{eq:wacon22}) is equivalent to
	\begin{equation*}\label{eq:jjj}
	(\hat{\theta}(x)-\hat{\theta}(y),x-y)\le-\lambda\|x-y\|^{2}
	\end{equation*}
	and thus to $\ominus\theta$ being $\lambda$-strongly monotone, which is impossible, as we saw in Corollary \ref{cor:ecnsd}.
\end{proof}
       
Finally, let us relate the previous theorem to the results of Subsection 5.1. If $A$ obeys (\ref{eq:e5}) with $\lambda=0$, then $\Lambda\equiv 0$, whence by Corollary \ref{cor:Aiinv} the Aitchison measure is invariant for the corresponding replicator diffusion $Y$. Moreover, in dimensions $n\ge 4$, $Y$ must be transient. Indeed, transience holds also for $n=2,3$. To see this, note by Proposition \ref{prop:NA}, $A$ can have interior NE only when $u\in\langle \mathbf{1}\rangle$, in which case the stochastic replicator dynamic degenerates to an undrifted Aitchison diffusion and is thus transient. Otherwise, if $A$ has no interior NE, transience follows from \citep[Cor. 4.16]{hofbauer2009time}.\par
   If $\lambda>0$, then $A$ satisfies in particular condition $(i)$ of Corollary \ref{cor: diri} (with $\lambda=|\alpha|$). Moreover $A$ has an interior NE iff either $u\in\langle\mathbf{1}\rangle$ in which case $p^{*}=e$ or $\lambda>|u|+n\max_{i}u^{-}_{i}$ by Proposition \ref{prop:hmpf}. If $p^{*}\in\Delta$ is a NE for $A$, it follows that the Dirichlet distribution with parameter $\lambda p^{*}$ is invariant for $Y$. In this case we observe that the stochastic replicator dynamic obeys the same contraction behavior as the Langevin dynamic in Proposition \ref{pro:dircon}.  \par 
       \vspace{1.1cm}

   \noindent
   A\footnotesize{CKNOWLEDGEMNTS}. \normalsize{The author is much obliged to Denis Serre, who provided the proof for the case $n\ge 4$ in $(iii)\implies(iv)$ of Theorem 21.} Much appreciated are also the critical remarks of Max von Renesse that helped to substantially improve the paper.

\bibliographystyle{plain}  
 \renewcommand\refname{\begin{center}
 		\normalfont{R\footnotesize{EFERENCES}}
 \end{center}}
\bibliography{exist}

\end{document}